\documentclass[reqno]{amsart}
\usepackage{graphicx, dsfont}
\usepackage{color}
\newtheorem{theorem}{Theorem}

\newtheorem{proposition}[theorem]{Proposition}

\newtheorem{lemma}[theorem]{Lemma}

\newtheorem{remark}{Remark}

\def\qed{\hbox{${\vcenter{\vbox{		 
   \hrule height 0.4pt\hbox{\vrule width 0.4pt height 6pt
   \kern5pt\vrule width 0.4pt}\hrule height 0.4pt}}}$}}

\def\cF{\mathcal F}

\def\bC{\mathbb C}

\def\bE{\mathbb E}

\def\bN{\mathbb N}

\def\bP{\mathbb P}

\def\bR{\mathbb R}

\def\bZ{\mathbb Z}

\def\stirl#1#2{\left\{\begin{matrix}#1 \\ #2 \end{matrix}\right\}}

\newcommand\symbo{\frac{(-i)^N \beta \lambda^N}{N!}}
\newcommand\dtime{\pmb D}

\begin{document}

\title[Solution of fractional high order heat-type equations]{Probabilistic representation formula for the solution of fractional high order heat-type equations}

\author{Stefano Bonaccorsi, Mirko D'Ovidio  and Sonia Mazzucchi} 

\date{\today}

\address{Stefano Bonaccorsi \newline
Dipartimento di Matematica, University of Trento
    Trento, via Sommarive 14, 38123 Povo (Trento), Italy
}
\email{stefano.bonaccorsi@unitn.it}

\address{Mirko D'Ovidio \newline
Dipartimento di Scienze di Base e Applicate per l’Ingegneria, Sapienza University of Rome, via A. Scarpa 10, 00161 Rome, Italy
}
\email{mirko.dovidio@uniroma1.it}

\address{Sonia Mazzucchi \newline
Dipartimento di Matematica, University of Trento
    Trento, via Sommarive 14, 38123 Povo (Trento), Italy
}
\email{ sonia.mazzucchi@unitn.it}


\begin{abstract}
We propose a probabilistic construction for the solution of a general class of fractional high order heat-type equations in the one-dimensional case, 
by using a sequence of random walks in the complex plane with a suitable scaling. A time change governed by a class of subordinated processes 
allows to  handle the fractional part of the derivative in space. 
We first consider evolution equations with space fractional derivatives of any order, and later we show the extension to equations with time fractional derivative (in the sense of Caputo derivative) of order $\alpha \in (0,1)$.
\end{abstract}

\maketitle

\section{Introduction}

The connection between partial differential equations and stochastic processes, or, more generally, functional integration, is an extensively developed theory{\color{black}, which covers one-dimensional, finite dimensional and infinite dimensional problems. Since this paper is mainly devoted to the one-dimensional case, we shall specialise this introduction to such problems}. 
The first and  main example is the Feynman-Kac formula \cite{Kac1,Kac2} providing the solution of the Cauchy problem for the  ``heat equation with sink''
\begin{equation}
\label{heat}
\partial_t u(t,x)=\frac{1}{2} \partial_x^2 u(t,x) +V(x) u(t,x),\qquad x\in \bR, t\in \bR^+,
\end{equation} 
in terms of a Wiener integral of the form 
\begin{align*}
u(t,x)=\bE\left[u(0,x+W(t))e^{\int_0^t V(x+W(s)) \, {\rm d}s}\right],
\end{align*}
where $W=(W(t))_{t\geq 0}$ denotes a $1$-dimensional Wiener process.\\
In fact the  connection between heat equation and Wiener process is just  a particular case of a general theory connecting Markov processes with parabolic equations associated to second order elliptic operators (see \cite{Dyn,Fre}).  
{\color{black}In the general, $d$-dimensional case, we are} given a Lipschitz  map $\sigma :\bR^d \to L(\bR^d,\bR^d)$ from $\bR^d$ to the $d\times d$ matrices, a Lipschitz vector field $b:\bR^d\to \bR^d$  and a $d-$dimensional Wiener process $W=(W(t))_{t\geq 0}$,  the solution of the Cauchy problem 
\begin{equation}\label{PDE-gen}\left\{ \begin{array}{l}
\frac{\partial}{\partial t}u(t,x)=\frac{1}{2}Tr[\sigma(x)\sigma^*(x)\nabla^2 u(t,x)]+\langle b(x), \nabla u(t,x)\rangle+V(x) u(t,x)\\
u(0,x)=u_0(x)\\
\end{array}\right. 
\end{equation}
 is related to the solution $X^x = \{X^x(t)\}_{t\geq 0}$  of the stochastic differential equation
\begin{equation}
\label{process-gen}
\left\{ 
\begin{array}{l}
dX^x(t) = b(X^x(t)) \, {\rm d}t+\sigma(X^x(t)) \, {\rm d}W(t), \\
X^x(0)=x,\qquad x\in \bR^d
\end{array}
\right. 
\end{equation}
by the probabilistic representation  formula:
\begin{equation}
\label{Fey-Kac-gen}
u(t,x)=\bE\left[  u(0, X^x(t))e^{\int_0^t V(X^x(s)) \, {\rm d}s}    \right], \quad t\geq 0, x\in \bR^d.
\end{equation}

{\color{black}One possible extension of the heat equation occours if we replace the right-hand side of \eqref{heat}}
by a spatial fractional derivative operator $\partial _x^\alpha$, namely by a Fourier operator with symbol $\psi (y)=(-iy)^\alpha$, {\color{black}which leads to} the equation 
\begin{equation}
\label{d-alpha}
\partial_t u(t,x)=-\partial _x^\alpha u(t,x), \qquad x\in \bR,\ t\in \bR^+,\ \alpha \in (0,1).
\end{equation}
{\color{black}Fractional powers of operators have been introduced in \cite{Boch49, Feller52} where the authors considered fractional powers	 of the Laplace operator. For a closed linear operator $A$, the fractional operator $(-A)^\alpha$ has been investigated by many researchers, the reader can consult the works \cite{ KraSob59, Balak60, Wata61, Kom66, HoWe72} for example.}
{\color{black}Equation \eqref{d-alpha}} is associated with a L\'evy process $\{H^\alpha(t)\}_{t\geq 0}$ called {\em $\alpha$-stable subordinator} for $\alpha \in (0,1)$
(see \cite{Ber97} and Appendix A). In particular  the Laplace transform of  $H^\alpha(t)$ has the form 
$$\bE[e^{-\lambda H^\alpha(t)}]=e^{-t\lambda^\alpha},\qquad \lambda \in \bR^+,$$
hence the solution of the Cauchy problem associated to Eq. \eqref{d-alpha} is given by
\begin{equation}\label{rap-H}u(t,x)=\bE\left[  u(0, x+H^\alpha(t))\right].\end{equation}

The generalization of these results to different types of PDEs which do not satisfy the maximum principle is in general not possible \cite{DalFo}. 
In particular a probabilistic representation of the form \eqref{Fey-Kac-gen} or \eqref{rap-H}, giving  the solution in terms of the expectation with respect to a {\em real} stochastic process with independent increments, cannot be proved in case the second order elliptic operator on the right hand side of Eq. \eqref{PDE-gen} is replaced by a differential operator of order $N>2$, obtaining a high-order heat-type equation of the form 
\begin{equation}\label{eq-N}
\partial_t u(t,x)=\frac{\beta }{N!}\partial _x^Nu(t,x), \qquad x\in \bR,\ t\in \bR^+ ,
\end{equation}
where $\beta $ is a real constant {\color{black}satisfying some conditions on the sign, while the $\frac{1}{N!}$-term is the analog of the factor $\frac12$ for the heat equation}. In fact this no-go result was stated originally by Krylov \cite{Kry} in the case where $N=4$ and  is related to the non positivity of the  solution $g\equiv g(t,x)$ of the problem
\begin{equation}
\label{problem-g}
\left\lbrace \begin{array}{l}
\displaystyle \partial_t g = \frac{1}{N!} \partial_x^N g\\
\displaystyle g(x,0)=\delta(x), \quad x \in \mathbb{R}
\end{array} \right.
\end{equation}
as well as  the rather restricting conditions for the generalization of Kolmogorov existence theorem to the limit of a projective system of either {\em signed} or {\em complex} measures (see \cite{Tho} for this result and \cite{AlMa} for a discussion of its implication in the construction of a probabilistic representation for the solution of high-order PDEs).

 The problem of a probabilistic representation for the Cauchy problem associated to Eq. \eqref{eq-N}, namely {\em a generalized Feynman-Kac formula},  is extensively studied and   different approaches have been proposed, in particular in the case $N=4$. One of the first approaches was introduced by  Krylov \cite{Kry} and continued by Hochberg \cite{Hoc}, who introduced a stochastic pseudo-process whose transition probability function, defined as the solution of \eqref{problem-g},  is not positive definite.  The generalized Feynman-Kac formula is constructed in terms of  the expectation with respect to a  signed measure on $\bR^{[0,t]}$ with infinite total variation. For this reason the integral on   $\bR^{[0,t]}$ is not defined in Lebesgue sense, but  is meant as limit of finite dimensional cylindrical approximations \cite{BeHocOr}. It is worth  mentioning the work by D. Levin and T. Lyons \cite{LevLyo} on rough paths, conjecturing that the signed measure (with infinite total variation)  associated to the pseudo-process could exist on the quotient space of equivalence classes of paths corresponding to different parametrizations of the same path.
{\color{black}
Properties of the pseudo-process $X(t)$ associated with the signed measure $\bP$, corresponding to the fundamental solution of \eqref{problem-g} via the identity $\bP_x(X(t) \in {\rm d}y) = p(t,x,{\rm d}y)$ were studied by several authors, in particular Hochberg \cite{Hoc}, Orsingher \cite{Ors19, Hoc8, Nik15}, Lachal \cite{Lac11}, Nishioka \cite{Nishioka1996, Nishioka2001}. It shall be noticed that, in the case $N=4$, paths of $X(t)$ are not continuous.
}
 \\
A different approach was proposed by Funaki \cite{Funaki1979} and continued by Burdzy \cite{Burdzy1995}, based on  the construction of a complex valued stochastic process with dependent increments, obtained  by a certain composition of two independent Brownian motions. 

Recently in \cite{BoMa15} a new approach has been proposed. Starting from the weak convergence of the scaled random walk on the real line $S_n(t):=\frac{1}{\sqrt n}\sum_{j=1}^{\lfloor nt\rfloor}\xi_j$ to the Wiener process $B(t)$ ($\xi_j$ for $j=1,...,n$ being independent identically distributed Bernoulli random variables such that $\bP(\xi_j=1)=\bP(\xi_j=-1)=1/2$), the solution of the heat equation can be written as the limit 
$u(t,x)=\lim_{n\to +\infty }\bE[u(0, x+S_n(t))]$ . This formula can be generalized to the case of  Eq.\eqref{eq-N} with $N\in \bN$ and $N>2$, by constructing a sequence of complex random walks $\{W_n^{N,\beta}(t)\}_{n\in\bN}$ as  $W_n^{N,\beta}(t):=\frac{1}{n^{1/N}}\sum_{j=1}^{\lfloor nt\rfloor}\xi_j$, where $\xi_j$ for $j=1,...,n$ is a sequence of independent identically distributed complex random variables uniformly distributed on the set of $N$-th roots of $\beta$.
In fact, if $N>2$, the particular scaling exponent $1/N$  appearing in the definition of $W_n^{N,\beta}(t)$  does not allow the weak convergence of the sequence of random variables $W_n^{N,\beta}(t)$. However,  for a suitable class of analytical initial data $u_0:\bC\to \bC$, the limit  of the expectation, namely $\lim_{n\to\infty}\bE\left[u_0(x+W^{N,\beta}_n(t))\right]$, exists  and provides a probabilistic representation for the solution of Eq. \eqref{eq-N}.
{\color{black}Properties of the random walks $W^{N,\beta}$ are further studied in \cite{BoCalMaz}, where a kind of It\^o calculus is introduced, by the construction of the It\^o integral and an It\^o formula for the  limit of these processes.}
{\color{black}
A similar approach to pseudo-processes was introduced by Lachal \cite{Lachal2013} for $N$ even: in that paper, the $\xi_j$'s take values in the discrete set $\{-N/2, -N/2+1, \dots, N/2-1, N/2\}$ with (positive or negative) real pseudo-probabilities $\bP(\xi = k) = \delta_{k=0} + (-1)^{k-1} \binom{N}{k+N/2}$. 
He proves that with the same scaling exponent as ours, his sequence of pseudo-random walks converges to the pseudo-process associated with the signed measure $\bP$ introduced before.}

{\color{black}An extension of Eq.\ \eqref{problem-g} to the case of higher order space fractional derivatives, by replacing the order $N\in\bN$ with the product $N\alpha$, with $N\in\bN$ and  $\alpha \in (0,1)$,
 has been obtained in \cite{OrTo}. 
The authors define  a sequence of pseudo random walks, converging weakly to pseudo-processes stopped at stable subordinators. 
The fundamental solution of higher order space fractional heat type equations is obtained as the limit of the (signed) laws of the pseudo random walks, which are signed measures. 
}

A related problem is the study of time fractional equations of the form 
\begin{equation}\label{t-frac}
\begin{aligned}
\dtime^\alpha_t u(t,x) = & A u(t,x) \\
u(0,x)= & f(x),
\end{aligned} 
\end{equation}
where $A = \frac{\beta}{N!}\partial _x^N$ and  the time-fractional derivative $\dtime^\alpha_t$ must be understood in the Caputo sense.
\\
The fractional diffusion equations are related to the so-called fractional and anomalous diffusions, that is,  diffusions in non-homogeneous media with random fractal structures, see for instance \cite{meerschaert-nane-xiao}.
The term fractional is due to the replacement of standard derivatives with respect to time $t$ with fractional derivatives and the corresponding equations describe delayed diffusions. However, we do not care about the geometrical structure of the medium and therefore our meaning of fractional diffusions is far from the definition of fractional diffusions introduced in \cite{BARL}.
Anomalous diffusion occurs when the mean square displacement (or time-dependent variance) is stretched by some index, in other words proportional to $t^\alpha$ for instance.
In the literature, equation \eqref{t-frac} is used as a mathematical model of a wide range of important physical
phenomena, usually named {\em sub-} or {\em super-diffusions}, for instance in microelectronics (dielectrics and semiconductors), polymers,
transport phenomena in complex systems 
and anomalous heat conduction in porous glasses  and random media (see for instance \cite{SKM, Miller1993, GioRom92, Gorenflo1997}).
\\
Fractional diffusion equations as \eqref{t-frac} where $N=2$ have been investigated by several researchers. In \cite{Koc89, Nig86, Wyss86} the authors  established  the mathematical foundations of fractional diffusions. In \cite{Wyss86, SWyss89} and later in \cite{Hil00, DOVsl}  the authors studied the solutions to the heat-type fractional diffusion equation and the corresponding representation of the solutions in terms of Fox's functions. 
The explicit representation of the solutions by means of stable densities have been studied in \cite{OB03,Orsingher2009} and, in the case $\alpha = 1/2^n$, Orsingher \cite{Orsingher2009} proved the connection of this solution with the distribution of $n$-iterated Brownian motion. 
\\
For a general operator $A$ acting in space, several results can also be listed. Nigmatullin \cite{Nig86}  gave a physical interpretation when $A$ is the generator of a Markov process. 
Zaslavsky \cite{Zasl94}  introduced the fractional kinetic equation for Hamiltonian chaos. The problem concerning  an infinitely divisible generator $A$ on a finite dimensional space has been investigated in \cite{BM01}. 
In general, a large  class of fractional diffusion equations are solved by time-changed stochastic processes. We usually refer to such processes as stochastic solutions to the driving equations. Stochastic solutions to fractional diffusion equations can be realized through time change by inverse stable subordinators, see for example,  \cite{BMN09ann, Orsingher2009}. 
Indeed, for a guiding process $X(t)$ with generator $A$ we have that $X(L^\alpha(t))$ is governed by $\partial^\alpha_t u= A u $ where the process $L^\alpha(t)$, $t>0$ is an inverse or hitting time process to a $\alpha$-stable subordinator. 
The time-fractional derivative comes from the fact that $X(L^\alpha(t))$ can be viewed as a scaling limit of continuous time random walks where the iid jumps are separated by iid power law waiting times (see \cite{MSheff04, Meerschaert2013}). 
The interested reader can find  a short survey on these results in \cite{NANERW}. Results on subordination principles for fractional evolution equations can be found in \cite{Boch49, Bazh00}.


Beside the interest in studying fractional equations, many researchers have concentrated their efforts toward the study of the higher-order counterpart \eqref{eq:intro-m1}
of fractional equations, see for example \cite{allouba1, BMN09, DeBla, dovSPL, KeyaLiza12, Nane2008}.  
When the underling operator generates  a strongly continuous semigroup, the time-changed process can be considered in  order to study the fractional diffusion equations and also, the higher-order equation with a non homogeneous term involving higher-order powers of the driving operator. The reader can consult Keyantuo and Lizama \cite{KeyaLiza12} and the references therein.

\vskip 1\baselineskip

The first aim of the present work is 
the generalization of the construction in \cite{BoMa15} to the case of higher order fractional derivatives of order $N\alpha$,  with $N\in\bN$ and  $\alpha \in (0,1)$. 
After a couple of sections where we introduce some preliminary results, mainly taken from \cite{BoMa15} and 
\cite{BoCalMaz}, in Section \ref{sez4}
we provide a probabilistic representation of the solution to a family of equations of the form 
\begin{equation}\label{eq-A}\partial_t u=-Au, \end{equation}
where $A:D(A)\subset L^2(\bR)\to L^2(\bR)$ is a Fourier integral operator with symbol $\Psi:\bR\to\bC$ either  of the form $\Psi (y)= c(iy)^{N\alpha}$ or $\Psi(y)=c'|y|^{N\alpha}$, with $c,c'\in\bC$ suitable constants. 

As opposite to \cite{OrTo}, in our approach  the solution of the  equation is given in terms of the limit of expectations with respect to the probability laws of rather simple jump processes in the complex plane, without the need to introduce {signed probabilities}. 

By subordinating the sequence of complex random walks $W_n^{N,\beta}(t)$  associated to the $N$-th order equation \eqref{eq-N} with a sequence of processes $\{S^\alpha_m(t)\}_{m\in\bN}$ converging weakly as $m\to+\infty$ to the $\alpha$-stable process $H^\alpha(t)$, a sequence $\{X_{n,m}(t)\}_{(n,m)\in \bN^2}$ of jump processes on the complex plane is defined as $X_{n,m}(t):=W_n^{N,\beta}(S_m^\alpha(t))$. It converges formally to an $N\alpha$-stable process in the sense that
$$\lim_{m\to\infty}\lim_{n\to\infty}\bE\left[e^{iyX_{n,m}(t)}\right] =e^{-t\left((-1)^{N+1}i^N \frac{\beta}{N!}y^N\right)^\alpha}, \qquad y\in \bR.$$
This result allows the representation of the solution  to \eqref{eq-A} with $\widehat{Au}(y):=\left(\frac{(-1)^{N+1}i^N\beta y^N}{N!}\right)^\alpha\hat u(y)$ (  $\hat \,$ denoting the  Fourier transform) as the limit
\begin{equation}
\lim_{m\to\infty}\lim_{n\to \infty}\bE\left[u_0\left(x+ W_n^{N,\beta}(S_m^\alpha(t)) \right)\right] ,\qquad t\in\bR^+, x\in\bR,
\end{equation}
for a suitable class of analytical initial data $u_0$. Moreover we show that in the case the symbol of the operator $A$ has the form $\Psi(y)=|y|^{N\alpha}$,  the solution of \eqref{eq-A} can still be given by a formula
$$\lim_{m\to\infty}\lim_{n\to \infty}\bE\left[u_0\left(x+X_{n,m}(t)+\tilde X_{n,m}(t)\right)\right] ,\qquad t\in\bR^+, x\in\bR,$$
where $X_{n,m}(t)$ and $\tilde X_{n,m}(t)$ are two independent copies of the process $X_{n,m}(t)=W_n^{N,\beta}(S^\alpha_m(t))$ and $\tilde X_{n,m}(t)=W_n^{N,\beta'}(S^\alpha_m(t))$, constructed respectively by setting $\beta =N!$ and $\beta'=-N!$.


In Section \ref{sez6} we also consider time fractional equations of the form 
\eqref{t-frac}
and we
prove that the solution of the initial value problem, for a suitable class of initial data $f$, is given by
\begin{equation}
\label{sol-t-frac}
u(t,x)=\lim_{n\to+\infty}\bE\left[f(x+W_n^{N,\beta}(L^\alpha(t)))\right],
\end{equation}
where $L^\alpha(t)$ is the inverse of the subordinator $H^\alpha(t)$.
\\
Furthermore, 
in the case where $\alpha = M^{-1}$ for some $M \in \bN$, $M > 1$, 
we prove that problem \eqref{t-frac} is equivalent to the diffusion equation with non-local forcing term
of the form
\begin{equation}
\label{eq:intro-m1}
\begin{aligned}
\partial_t u(t,x) &= A^{1/\alpha} u(t,x) + \sum_{k=1}^{1/\alpha-1} \frac{1}{\Gamma(\alpha k)} t^{\alpha k -1} A^k f(x)
\\
u(0,x) &= f(x)
\end{aligned}
\end{equation}
in the sense that both problems share the same solution \eqref{sol-t-frac}.

\section{A sequence of random walks on the complex plane}\label{sez-2}
The present  section is devoted to the construction of a sequence of random walks in the complex plane whose limit can be interpreted in a very weak sense (see theorem \ref{teo3}) as an $N-$ stable stochastic process, with $N\in \bN$.\\
Let $(\Omega, \cF, \bP)$ be a probability space. Let $\beta$ be a complex number and $N \geq 2$ a given integer. 
\\
Let $R(N)= \{e^{2 i \pi k/N},\ k=0,1,\dots, N-1\}$ denote  the set of the $N-$ roots of the unity and let us consider 
 the complex random variable $\xi_{N,\beta}$ that has uniform distribution on the set $\beta^{1/N} R(N)$, namely for any function $f:\bC\to \bC$:
\begin{equation}
\label{e2}
\bE[f(\xi_{N,\beta})] = \frac{1}{N} \, \sum_{k=0}^{N-1} f(\beta^{1/N} e^{2 i \pi k/ N}).
\end{equation}

The random variable $\xi_{N,\beta}$ has some interesting properties (see \cite{BoMa15} for detailed calculations). Indeed it admits (complex) moments of any order:
\begin{equation}
\label{e3}
\bE[\xi_{N,\beta}^m] =
\begin{cases}
\beta^{m/N}, & m = n N,\ n \in \bN,
\\
0, & \text{otherwise}
\end{cases}
\end{equation}
In particular, its characteristic function is
\begin{align*}
\psi_{\xi_{N,\beta}}(\lambda) = \frac{1}{N} \, \sum_{k=0}^{N-1} \exp(i\beta^{1/N} \lambda e^{2 i \pi k/ N}).
\end{align*}
Further, we may compute the absolute moments of $\xi_{N,\beta}$ obtaining
\begin{align*}
\bE[|\xi_{N,\beta}|^m] = |\beta|^{m/N}.
\end{align*}

Equation \eqref{e3} is the starting point for the construction of a particular sequence of random variables on the complex plane which converges (in a sense that will be explained soon) to a stable random variable of order $N\geq 2$. \\ 
Let $\{\xi_j,\ j \in \bN\}$ be a sequence of i.i.d.\ random variables having uniform distribution on the set $\beta^{1/N} R(N)$ as in \eqref{e2}.
Let $S(N,\beta)_n$ be the random walk defined by the $\{\xi_j\}$, i.e.,
\begin{align*}
S(N,\beta)_n = \sum_{j=1}^n \xi_j.
\end{align*}

Interesting properties of the complex random walk $S(N,\beta)_n$ in the case $\beta=1$ have been investigated in \cite{BoMa15}. \\
In the case where $N=3$ and  the walk $S(3,1)_n$ occurs on the
regular lattice generated by the vectors $\{(1,0),\
(-\frac12,\frac{\sqrt{3}}{2}),\ (-\frac12,-\frac{\sqrt{3}}{2})\}$,
considered as a {\em directed} graph.
Therefore, the motion is $3$-periodic, and a return to the origin only
happens if the same number of steps is made in every direction.
Therefore, we compute
\begin{align*}
\bP(S(3,1)_{3m}=0) = \frac{(3m)!}{(m!)^3} \frac{1}{3^{3m}}
\end{align*}
and Stirling's formula implies
$\bP(S(3,1)_{3m}=0) \sim \frac{1}{2 \, \pi \, m}$;
hence the expected number of returns to the origin is
\begin{align*}
\sum_{m=1}^\infty \bP(S(3,1)_{3m}=0) \sim \sum_{m=1}^\infty \frac1m = +\infty
\end{align*}
and the process is recurrent.\\
The case $N=4$ corresponds to the standard, two-dimensional random walk,
hence the motion is $2$-periodic  on the lattice
$\bZ^2$ (this time considered as an {\em undirected} graph). Finally, the
motion is recurrent.\\
In the case where $N=5$ the process is transient. Indeed, in this case,
the motion is again $5$-periodic and the only way to return to the origin
is taking the same number of steps in each direction.
Hence, again by an application of Stirling's formula,
\begin{align*}
\bP(S(5,1)_{5m}=0) = \frac{(5m)!}{(m!)^5} \frac{1}{5^{5m}} \sim
\frac{\sqrt{5}}{(2 \, \pi \, m)^2}
\end{align*}
and the expected number of returns is finite:
\begin{align*}
\sum_{m=1}^\infty \bP(S(5,1)_{5m}=0)  \sim \sum_{m=1}^\infty \frac{\sqrt{5}}{(2
\, \pi \, m)^2} < \infty
\end{align*}
However for $N>5$ the following result holds \cite{BoMa15}.
\begin{proposition}
Let $N\ge 5$. The process $\{S(N,1)_n\}$ is neighborhood-recurrent, i.e., for
every $x$ in the lattice generated by the basis $\{\beta^{1/N} e^{2 \pi i
k/N},\ k=0,1,\dots,N-1\}$ it holds
\begin{align*}
\bP(|S(N,1)_n - x| \le \varepsilon \quad \text{infinitely often}) = 1.
\end{align*}
\end{proposition}
By the classical central limit theorem the sequence  $\frac{1}{n^{1/2}} S(N,\beta)_n$ converges to a Gaussian random variable for any $N\in \bN$ and $\beta\in \bC$. However, if we consider a normalized random walk $\tilde S(N,\beta)_n$ defined as
\begin{align*}
\tilde S(N,\beta)_n = \frac{1}{n^{1/N}} S(N,\beta)_n
\end{align*}
then we have convergence to a stable distribution of order $N$ in the sense that (see \cite[Theorem 2]{BoMa15})
\begin{equation}\label{conv-S-n}
\lim_{n \to \infty} \bE[\exp(i \lambda \tilde S(N,\beta)_n)] = \exp \left(\frac{i^N \beta}{N!} \lambda^N \right).
\end{equation}

It is important to remark that for $N>2$ the sequence $\tilde S(N,\beta)_n$ cannot converge in distribution because of the scaling exponent $1/N$. 
In fact, for $N > 2$, the function $\exp(c x^N)$ is not a well defined characteristic function. 
Equation \eqref{conv-S-n} states that even if the distributions of $\tilde S(N,\beta)_n$ cannot converge weakly to a measure, the integral of suitable functions (the exponentials) with respect to these measures has a well defined limit as $n\to \infty$.

It is possible to extend the definition of $\tilde S(N,\beta)_n$ to a continuous time process and construct a sequence of jump processes  $W^{N,\beta}_n(t)$ on the complex plane such that  $\tilde S(N,\beta)_n =W^{N,\beta}_n(1)$. Given  a sequence  $\{\xi_j\}$ of independent copies of the random variable $\xi_{N,\beta}$ defined in \eqref{e2}, let us consider for $t>0$ the process $W^{N,\beta}_n(t)$ defined by 
\begin{align}
W^{N,\beta}_n(0) &= 0;\nonumber \\
W^{N,\beta}_n(t) &=  \frac{1}{n^{1/N}} \sum_{j=1}^{\lfloor nt\rfloor} \xi_j = \frac{1}{n^{1/N}} S(N,\beta)_{\lfloor nt \rfloor}. \label{Wn} 
\end{align}

The  process $W^{N,\beta}_n$ has some interesting properties. The following lemma shows the particular behavior of the complex moments.

\begin{lemma}\label{th-momenti}
For $k \in \bN$ and  $t\in \bR^+$,  the $k$-moment of $W^{N,\beta}_n(t)$ satisfies
\begin{align*}
\bE[(W^{N,\beta}_n(t))^k] =
\begin{cases}
\left(\frac{\beta t}{N!}\right)^{k/N} \frac{k!}{(k/N)!}{\bf 1}_{[0,\lfloor nt\rfloor]}(k/N) + R(n,k{\color{black};t}),  & k = h N,\ h \in \bN, \, 
\\
0, & \text{otherwise}
\end{cases}
\end{align*}
(${\bf 1}_{[0,\lfloor nt\rfloor]}$ being the indicator function of the interval $[0,\lfloor nt\rfloor]$).
\\
{\color{black}For $k = 0$ and $k= N$, i.e., for $h = 0$ and $h=1$, the remainder term vanishes.}
For $k=hN$, $h\in \bN$, $h \ge 2$, the remainder term $R(n,k{\color{black};t})$ satisfies the following estimate:
\begin{align*}
|R(n,hN{\color{black};t})|\leq \frac{|\beta|^ht^{h-1}(h^2+h)}{2n}\frac{(hN)!}{h!(N!)^h}+\frac{|\beta|^h t^{h-1}}{n}\left(\frac{0.792hN}{\log(hN+1)}\right)^{hN}.
\end{align*}
\end{lemma}

\begin{proof}
Let $W^{N,\beta}_n(t)=\frac{1}{n^{1/N}}\sum_j^{\lfloor nt\rfloor}\xi_j$ and $\psi_n$ be its characteristic function, namely:
$$\psi_n(\lambda):=\bE[e^{i\lambda W^{N,\beta}_n(t)}]$$
We have that 
$$\bE[(W^{N,\beta}_n(t))^k] =(-i)^k\frac{d^k\psi_n}{d\lambda^k} (0),$$
where $\psi_n$ is equal to
\begin{align*}
\psi_n(\lambda) = \left( \bE\left[\exp(\frac{1}{n^{1/N}} i \lambda \xi)\right] \right)^{\lfloor nt\rfloor} = \left( \psi_\xi\left(\frac{\lambda}{n^{1/N}}\right) \right)^{\lfloor nt\rfloor},
\end{align*}
where $\psi_\xi$ is the characteristic function of $\xi$.\\
By Fa\'a di Bruno's formula \cite{KraPar}
\begin{equation}\label{Faa1}
\frac{d^k\psi_n}{d\lambda^k} (\lambda)=\sum_{\pi\in \Pi}C(|\pi|, \lambda)\prod_{B\in \pi}\left(\frac{\psi_\xi^{(|B|)}(\lambda/n^{1/N})}{n^{|B|/N}}\right)
\end{equation}
where 
$\pi$ runs over the set $\Pi$ of all partitions of the set $\{1,...,k\}$,  
$B\in \pi$ means that the variable $B$ runs through  the list of 
 blocks of the partition $\pi$, 
$|\pi|$ denotes the number of blocks of the partition $\pi$ and $|B|$ is the cardinality of a set $B$, while the function $C:\bN\times \bR\to\bC$ is equal to
\begin{align*}
C(j,\lambda) =
\begin{cases}
\frac{\lfloor nt\rfloor!}{(\lfloor nt\rfloor-j)!}\left(\psi_\xi(\lambda/n^{1/N})\right)^{\lfloor nt\rfloor-j},  & \lfloor nt\rfloor\geq j \, 
\\
0, & \text{otherwise}
\end{cases}
\end{align*} 
 Formula \eqref{Faa1} can be written in the equivalent form:
\begin{multline}\frac{d^k\psi_n}{d\lambda^k} (\lambda)=
\\\sum \frac{k!}{m_1!m_2! \cdots m_k!}\frac{\lfloor nt\rfloor!}{(\lfloor nt\rfloor-(m_1+m_2+\dots+m_k))!}\left(\psi_\xi(\lambda/n^{1/N})\right)^{\lfloor nt\rfloor-(m_1+m_2+\dots+m_k)}\Pi_{j=1}^k\left(\frac{\psi_\xi^{(j)}(\lambda/n^{1/N})}{j!n^{j/N}}\right)^{m_j}\end{multline}
where the sum is over the $k-$tuple of non-negative integers $(m_1,m_2,...,m_k)$ such that $m_1+2m_2+\dots +km_k=k$ and $m_1+m_2+...+m_k\leq \lfloor nt\rfloor$.
In particular we have:
\begin{equation}\label{Faa2}
\frac{d^k\psi_n}{d\lambda^k} (0)=\sum_{\pi\in \Pi}\frac{\lfloor nt\rfloor!}{(\lfloor nt\rfloor-|\pi|)!}\Pi_{B\in \pi}\left(\frac{\psi_\xi^{(|B|)}(0)}{n^{|B|/N}}\right)
\end{equation}
 where the first sum runs over the partitions $\pi$ such that $|\pi|\leq \lfloor nt\rfloor $ or equivalently
\begin{equation}\label{FaadiBruno}
\frac{d^k\psi_n}{d\lambda^k} (0)=\sum \frac{k!}{m_1!m_2! \cdots m_k!}\frac{\lfloor nt\rfloor!}{(\lfloor nt\rfloor-(m_1+m_2+\dots+m_k))!}\Pi_{j=1}^k\left(\frac{\psi_\xi^{(j)}(0)}{j!n^{j/N}}\right)^{m_j}.
\end{equation}
Since $\psi_\xi^{(j)}(0)=(i)^j\bE[\xi^j]$, and $\bE[\xi^j]\neq 0$ iff $j=mN$, with $m\in \bN$, then 
product $\Pi_{j=1}^k\left(\frac{\psi_\xi^{(j)}(0)}{j!n^{j/N}}\right)^{m_j}$ is non vanishing iff $m_j=0$ for $j\neq lN$ and $k=Nm_N+2Nm_{2N}+...$, 
i.e. if $k$  is a multiple of $N$. 
Analogously in the sum appearing in formula \eqref{Faa2} the only terms giving a non vanishig contribution correspond to those partitions $\pi$ having blocks $B$ with a number of elements which is a multiple of $N$, giving, for $k=hN$:
\begin{equation}\label{Faa3}
\frac{d^{hN}\psi_n}{d\lambda^{hN}} (0)=i^{hN}\frac{\beta^h}{n^h}\sum_{\pi\in \Pi}\frac{\lfloor nt\rfloor!}{(\lfloor nt\rfloor-|\pi|)!}
\end{equation}
where again the sum runs over the partitions $\pi$ such that $|\pi|\leq \lfloor nt\rfloor $.
Equivalently:
\begin{eqnarray*}
\frac{d^{hN}\psi_n}{d\lambda^{hN}} (0)&=&\sum \frac{(hN)!}{(m_N)!(m_{2N})! \cdots (m_{hN})!}\frac{\lfloor nt\rfloor!}{(\lfloor nt\rfloor-(m_N+m_{2N}+\dots+m_{hN}))!}\Pi_{l=1}^h\left(\frac{\psi_\xi^{(lN)}(0)}{(lN)!n^{l}}\right)^{m_{lN}},\\
&=&\sum \frac{(hN)!}{(m_N)!(m_{2N})! \cdots (m_{hN})!}\frac{\lfloor nt\rfloor!}{(\lfloor nt\rfloor-(m_N+m_{2N}+\dots+m_{hN}))!}\Pi_{l=1}^h\left(\frac{i^{lN}\beta ^l}{(lN)!n^{l}}\right)^{m_{lN}},\\
&=&\frac{i^{hN}\beta^{h}}{n^{h}}\sum \frac{(hN)!}{(m_N)!(m_{2N})! \cdots (m_{hN})!}\frac{\lfloor nt\rfloor!}{(\lfloor nt\rfloor-(m_N+m_{2N}+\dots+m_{hN}))!}\Pi_{l=1}^h\frac{1}{((lN)!)^{m_{lN}}},
\end{eqnarray*}
where the sum is over the $h$-tuple of non-negative integers $(m_N,m_{2N},...,m_{hN})$ such that $m_N+2m_{2N}+...+hm_{hN}=h$ and $m_N+m_{2N}+...+m_{hN}\leq \lfloor nt\rfloor$.\\
Hence, we have 
$$
\bE[(W^{N,\beta}_n(t))^{hN}]=\beta ^{h}\sum \frac{(hN)!}{(m_N)!(N!)^{m_N}(m_{2N})!(2N)!^{m_{2N}} \cdots (m_{hN})!(hN)!^{m_{hN}}}\frac{\lfloor nt\rfloor!}{n^{h}(\lfloor nt\rfloor-(m_N+m_{2N}+\dots+m_{hN}))!}
$$
When $n\to \infty$, the leading term in the previous sum is the one corresponding to $m_N=h$ (hence $m_2N=...=m_{hN}=0)$, which is equal to 
$$\beta^h\frac{(hN)!}{(m_N)!(N!)^{h}}\frac{\lfloor nt\rfloor!}{n^{h}(\lfloor nt\rfloor-h)!}=\beta^ht^h\frac{(hN)!}{h!(N!)^{h}}+\beta^h\frac{(hN)!}{h!(N!)^{h}}\left(\frac{\lfloor nt\rfloor!}{n^{h}(\lfloor nt\rfloor-h)!}-t^h\right)$$
In the case where $\lfloor nt \rfloor <h$ then this term does not appear in the sum and  we can set it equal to 0. In the case where $\lfloor nt \rfloor \geq h$, we can estimate 
the quantity inside brackets as:
\begin{eqnarray*}
& &\left|\frac{\lfloor nt\rfloor!}{n^{h}(\lfloor nt\rfloor-h)!}-t^h\right|=
\frac{1}{n^h}\left|-(nt)^h+\Pi_{j=0}^{h-1}(\lfloor nt\rfloor-j)\right|
=\frac{1}{n^h}\Big|\Pi_{j=0}^{h-1}\big((\lfloor nt\rfloor-j)+(\{ nt\}+j)\big)-\Pi_{j=0}^{h-1}(\lfloor nt\rfloor-j)\Big|\\
& &\leq \frac{1}{n^h}\sum_{j=0}^{h-1}(\{ nt\}+j)\Pi_{k\neq j}nt=\frac{(nt)^{h-1}}{n^h}\sum_{j=0}^{h-1}(\{ nt\}+j)\leq \frac{(nt)^{h-1}}{n^h}\sum_{j=0}^{h-1}(1+j)=\frac{t^{h-1}(h^2+h)}{2n}
\end{eqnarray*}
where in the second line we have used that if $a_j,b_j\in \bR$, with $a_j,b_j\geq 0$ for all $j=0,...,m$, then (see Appendix)
$$\Pi_{j=0}^m(a_j+b_j)-\Pi_{j=0}^m a_j\leq \sum_{j=0}^m b_j\Pi_{k\neq j}(a_k+b_k).$$
Hence 
\begin{align*}
|R_1(n,h{\color{black};t})| =&\left|\beta^h\frac{(hN)!}{(m_N)!(N!)^{h}}\frac{\lfloor nt\rfloor!}{n^{h}(\lfloor nt\rfloor-h)!}-\beta^ht^h\frac{(hN)!}{h!(N!)^{h}}\right|
\leq  \frac{|\beta|^ht^{h-1}(h^2+h)}{2n}\frac{(hN)!}{h!(N!)^h}.
\end{align*}
By using formula \eqref{Faa3}, the remaining terms in the sum (corresponding to the $h$-tuple $(m_N,m_{2N},...,m_{hN})$ with $m_N<h$) are bounded by
\begin{align*}
R_2(n,h{\color{black};t})= \frac{\beta^h}{n^h}\sum_{\pi\in \Pi}\frac{\lfloor nt\rfloor!}{(\lfloor nt\rfloor-|\pi|)!}-\beta^h\frac{(hN)!}{(m_N)!(N!)^{h}}\frac{\lfloor nt\rfloor!}{n^{h}(\lfloor nt\rfloor-h)!} 
\leq  \frac{\beta^h}{n^h}\sum_{\pi\in \Pi}\lfloor nt\rfloor^{h-1}=\frac{\beta^h}{n}t^{h-1}B_{hN}
\end{align*}
where $B_{hN}$ is the Bell number, i.e. the number of partitions of the set $\{1,...,hN\}$. In particular, for any $h\in \bN$ (see \cite{BeTa}) $ B_{hN}<\left(\frac{0.792hN}{\log(hN+1)}\right)^{hN}$, hence 
$$|R_2(n,h{\color{black};t})|\leq \frac{|\beta|^h t^{h-1}}{n}\left(\frac{0.792hN}{\log(hN+1)}\right)^{hN}.$$
\end{proof}

A direct consequence of Lemma \ref{th-momenti} is the following theorem, which generalizes formula \eqref{conv-S-n} to the sequence of random walks $W^{N,\beta}_n$. 
\begin{theorem}\label{teo3}
For any $\beta\in\bC$ and $N\in \bN$, $N\geq 2$, the sequence of random walks  $W^{N,\beta}_n$ converges weakly  to a $N$-stable process in the sense that for any $t\geq 0 $ and $\lambda \in \bR$ the following holds:
\begin{equation}\label{conv-Wn}
\lim_{n \to \infty} \bE[\exp(i \lambda W^{N,\beta}_n(t))] = \exp \left(\frac{i^N \beta t}{N!} \lambda^N \right).
\end{equation}
\end{theorem}
\begin{proof}
\begin{align*}
\lim_{n \to \infty} \bE[\exp(i \lambda W^{N,\beta}_n(t))] &=\lim_{n \to \infty} \lim_{m \to \infty}\sum_{k=0}^m\frac{1}{k!}i ^k\lambda^k \bE[(W^{N,\beta}_n(t))^k]\\
&=\lim_{n \to \infty}\sum_{h=0}^{\lfloor nt\rfloor}\frac{i ^{hN}\lambda^{hN} }{(hN)!}\left(\frac{\beta t}{N!}\right)^h\frac{(hN)!}{h!}+ \lim_{n \to \infty} \lim_{m \to \infty}\sum_{h=2}^m \frac{i ^{hN}\lambda^{hN} }{(hN)!}R(n,hN{\color{green};t})
\\
&=\exp \left(\frac{i^N \beta t}{N!} \lambda^N \right)
\end{align*}
Indeed for $m,n\in \bN$
$$
\left|\sum_{h=2}^m \frac{i ^{hN}\lambda^{hN} }{(hN)!} R(n,hN{\color{black};t})\right|\leq \frac{C}{n},
\qquad C:= \sum_{h=2}^\infty t^{h-1}\frac{|\beta|^h|\lambda|^{hN} }{(hN)!}\left(\frac{(h^2+h)(hN)!}{2h!(N!)^h}+\left(\frac{0.792hN}{\log(hN+1)}\right)^{hN}\right)<\infty.
$$
\end{proof}

\begin{theorem}\label{th-limit}
Let $f:\bC\to \bC$ be an entire analytic function with the power series expansion $f(z)=\sum_{k=0}^\infty a_kz^k$, such that   
the coefficients $\{a_k\}$ satisfy the following condition:
\begin{equation} \label{coefficient}
\sum_{h=2}^\infty |a_{hN}|c^h\left(\frac{hN}{\log(hN+1)}\right)^{hN}<\infty\qquad \forall c\in \bR^+
\end{equation}
Then 
\begin{align*}
\lim_{n\to\infty}\bE[f(W^{N,\beta}_n(t))] = \sum_{h=0}^\infty a_{hN}\frac{(hN)!}{h!}\left(\frac{\beta t}{N!}\right)^h
 = \sum_{h=0}^\infty \frac{f^{(hN)}(0)}{h!}\left(\frac{\beta t}{N!}\right)^h.
\end{align*}
\end{theorem}

\begin{proof}
\begin{align*}
\lim_{n \to \infty} \bE[f(W^{N,\beta}_n(t))] &= \lim_{n \to \infty} \sum_{k=0}^\infty a_k\bE[(W^{N,\beta}_n(t))^k]
\\
&=\lim_{n \to \infty}\sum_{h=0}^{\lfloor nt\rfloor}a_{hN}\left(\frac{\beta t}{N!}\right)^h\frac{(hN)!}{h!}+\lim_{n \to \infty} 
\sum_{h=2}^\infty a_{hN}R(n,hN{\color{black};t})
\\
&=\sum_{h=0}^{\infty}a_{hN}\left(\frac{\beta t}{N!}\right)^h\frac{(hN)!}{h!}
\end{align*}
Indeed, by assumption \eqref{coefficient}, we have 
$$\left|\sum_{h=2}^\infty a_{hN}R(n,hN,t)\right|\leq \frac{C}{n}, \qquad C:=\sum_{h=2}^\infty |a_{hN}|\left(\frac{|\beta|^ht^{h-1}(h^2+h)}{2}\frac{(hN)!}{h!(N!)^h}+|\beta|^h t^{h-1}\left(\frac{0.792hN}{\log(hN+1)}\right)^{hN}\right)<\infty$$
where the series on the right hand side is convergent thanks to condition \eqref{coefficient}.
\end{proof}

\begin{remark}
\label{lemmacrescita}
Let us discuss further the assumption \eqref{coefficient}.

\begin{enumerate}
\item First, we provide the following simple, yet widely applicable, condition about the coefficients $\{a_k\}$:
\begin{equation}
\label{eq:lemmacrescita}
\mbox{
\begin{minipage}{.9\textwidth}
there exist $C_1,C_2\in\bR$ such that for all $k$ the  coefficients $a_k$ satisfy the inequality $|a_k|\leq \frac{C_1\,C_2^k}{k!}$.
\end{minipage}
}
\end{equation}
Then, if condition \eqref{eq:lemmacrescita} holds, the sequence of coefficients $\{a_k\}$ 
satisfies assumption \eqref{coefficient}.

\item Recall that an analytic function $f$ is said of {\em exponential type $c$} if 
$f(z)=\sum_{k=0}^\infty a_k z^k$ and $(a_k k!)^{1/k}\to c$ as $k\to \infty$.
\item If $f$ is of exponential type, then it satisfies assumption \eqref{coefficient}, since \eqref{eq:lemmacrescita} holds.

\item If $f:\bC\to\bC$ is the Fourier transform of a complex bounded variation measure $\mu$ on $\bR$ with compact support, then $f$ is of exponential type,
hence, 
in particular, 
it satisfies assumption \eqref{coefficient}. 

\item If $f:\bC\to\bC$ is the Fourier transform of a complex bounded variation measure $\mu$ on $\bR$ with compact support, i.e.
\begin{align*}
f(x)=\int_\bR e^{ixy} \, {\rm d}\mu (y),
\end{align*}
then for all $t\in\bR^+, x\in \bR$ it holds
\begin{align}
\label{eq:cor-6}
\lim_{n\to \infty }\bE[f(x+W^{N,\beta}_n(t))] = \int e^{iyx} e^{i^n\beta t \frac{y^N}{N!}} \, {\rm d}\mu(y).
\end{align}
\end{enumerate}
\end{remark}

\section{A sequence of subordinated processes}

Given a positive integer $N\in \bN$ with  $N\geq 2$ and a constant $\alpha \in (0,1)$, in the present section we construct by means of Bochner's subordination a sequence of jump processes on the complex plane converging weakly (in the sense of theorem \ref{teo13} stated below) to an $N\alpha $-stable process.  
Our aim is the derivation of the limit of subordinated processes $W_n^{N,\beta}(H^\alpha(t))$, where $W_n^{N,\beta}$ is the sequence of complex random walks defined in section \ref{sez-2} and $H^\alpha (t)$, $\alpha \in (0,1) $ and $t\geq 0$, is the $\alpha $-stable subordinator.
\par 
The first step is the construction of a sequence $\{S^\alpha_m(t)\}_{m\in \bN}$ of compound Poisson processes (with finite moments of any order) approximating the $\alpha$-stable subordinator $H^\alpha(t)$ (see Appendix \ref{app-A} for the definition of $\alpha-$ stable processes).
\\
Let $S^\alpha_m(t)$ be defined by
\begin{align}\label{Sm}
S^\alpha_m(t) = \frac{1}{m} \sum_{j=0}^{X} Y_{j}
\end{align}
where $X \sim {\rm Po}(\lambda t m^{2\alpha})$ is a Poisson random variable of parameter $\lambda \, t \, m^{2\alpha}$, with $\lambda=(\Gamma (1-\alpha))^{-1}$, 
and $Y_j$ are independent identically distributed\footnote{For notational simplicity we do not write explicitly the dependence of $Y_j$ and $X$ on the index $m$} copies of the random variable $Y_{[m]}$, with density $f_m(y) = c_m \, y^{-\alpha-1} \, {\pmb 1}_{\left(\frac{1}{m},m^2\right)}(y)$, $c_m =  \frac{\alpha}{m^{\alpha}(1 - m^{-3\alpha})}$ (see appendix \ref{powerlaw}), $X$ and the $\{Y_j\}$ being independent as well.

\begin{theorem}
The sequence of random variables $\{S^\alpha _m(t)\}_m$ converges weakly to the $\alpha$-stable subordinator $H^\alpha(t)$:
\begin{equation}
\label{eq:approx_poisson_comp}
\lim _{m\to \infty}\bE\left[ e^{i y S^\alpha _m(t)} \right] =\exp \left(-t (-i \, y)^\alpha \right), \quad y\in \bR.
\end{equation}
\end{theorem}

\begin{proof}
\begin{align*}
\bE\left[ e^{i y S^\alpha _m(t)} \right] 
= \bE\left[ \left( \bE [e^{i \frac{y}{m} Y_j} ]\right)^{X} \right]
= \exp\left(- \lambda \, t \, m^{2\alpha} \left(1 - \bE\left[e^{i \frac{y}{m} Y_{[m]}} \right] \right) \right)
\\
= \exp\left(- \lambda \, t \, m^{2\alpha} \, c_m \, \int_{1/m}^{m^2} \left(1 - e^{i \frac{y}{m} z}  \right) \, z^{-\alpha - 1} \, {\rm d} z\right)
\end{align*}
by means of a change of variable $x = z/m$ we get
\begin{align*}
\bE\left[ e^{i y S^\alpha _m(t)} \right] 
= \exp\left(- \lambda \, t \, m^{2\alpha} \, c_m \, \int_{1/m^2}^{m} \left(1 - e^{i y x}  \right) \, x^{-\alpha - 1} m^{-\alpha-1+1} \, {\rm d} x\right)
\\
= \exp\left(- \lambda \, t \, m^{\alpha} \, c_m \, \int_{1/m^2}^{m} \left(1 - e^{i y x}  \right) \, x^{-\alpha - 1} \, {\rm d} x\right)
\end{align*}
Since $\lim_{m \to \infty}m^\alpha c_m=\alpha$, 
we eventually obtain
\begin{align*}
\lim_{m\to \infty}\bE\left[ e^{i y S^\alpha_m(t)} \right] 
=
 \exp\left(-  \, t  \, \frac{\alpha}{\Gamma(1-\alpha)}\int_{0}^{\infty} \left(1 - e^{i y x}  \right) \, x^{-\alpha - 1} \, {\rm d} x\right)
 = \exp \left(- t (-i \, y)^\alpha \right)
 \end{align*}
\end{proof}
\begin{remark}
Formula \eqref{eq:approx_poisson_comp} remains valid by replacing $iy$, for $y \in \bR$, with a complex variable $z\in\bC$ 
with ${\rm Re} (z)\leq 0$
\begin{align}\label{eq:approx_poisson_comp-gen}
\lim_{m\to \infty}\bE\left[ e^{z S^\alpha _m(t)} \right] 
=
 \exp\left(-  \, t  \, \frac{\alpha}{\Gamma(1-\alpha)}\int_{0}^{\infty} \left(1 - e^{z x}  \right) \, x^{-\alpha - 1} \, {\rm d} x\right)
 = \exp \left(- t (-z)^\alpha \right)
 \end{align}
\end{remark}
\begin{lemma}
There exists a constant $C(\alpha)\in\bR^+$ such that the 
moments of the compound Poisson process $S^\alpha _m(t)$ given by \eqref{Sm} satisfy the following 
estimate 
\begin{align}\label{eq:st_mom_comp-2}
\bE[(S^\alpha _m(t))^k] &\le C(\alpha)^k t^k m^{k+2\alpha k-3\alpha}  \left(\frac{ k+1}{\log(k+2)}\right)^{k+1}
\end{align}
for every $k \ge 1$ and for every $m\geq (\Gamma (1-\alpha)/t)^{1/2\alpha}$ 
\end{lemma}

\begin{proof}
\begin{align*}
\bE[(S^\alpha _m(t))^k] = \bE \left[ \frac{1}{m^k} \sum_{j_1 = 0}^{X} \dots \sum_{j_k = 0}^X Y_{j_1} \dots Y_{j_k} \right]
\end{align*}
so we can rearrange the expectations and the sums as follows
\begin{align*}
\bE[(S^\alpha _m(t))^k] = \frac{1}{m^k} \sum_{x=0}^\infty \bP(X = x) \left( \sum_{j_1,\dots , j_k = 0}^{x}\bE\left[Y_{j_1} \dots Y_{j_k} \right]\right)
\end{align*}

By the estimate \eqref{eq:st_mom_Y_m-app} (see appendix \ref{app-B}) we obtain that for any choice of indexes $j_1,\dots , j_k = 0,...,x$ the following inequality holds:
$$\bE\left[Y_{j_1} \dots Y_{j_k} \right]\leq c(\alpha)^k m^{2k-3\alpha}$$
where $c(\alpha)=1\vee \frac{\alpha}{1-\alpha}$, hence, by estimate \eqref{eq:st_Poisson}
: 
\begin{align*}
\bE[(S^\alpha _m(t))^k] &\leq c(\alpha)^k m^{k-3\alpha} \sum_{x=0}^\infty \bP(X = x) (x+1)^k\\ 
 &= c(\alpha)^k m^{k-3\alpha}\frac{\Gamma(1-\alpha)}{ t m^{2\alpha}} \sum_{x=1}^\infty \bP(X = x) x^{k+1} \\
  &\le  c(\alpha)^k m^{k-3\alpha}\frac{\Gamma(1-\alpha)}{ t m^{2\alpha}}\left(1\vee \left(\frac{tm^{2\alpha}}{\Gamma (1-\alpha)}\right)^{k+1}\right)B_{k+1} \\
&\le  c(\alpha)^k m^{k-3\alpha}\frac{\Gamma(1-\alpha)}{ t m^{2\alpha}}\left(1\vee \left(\frac{tm^{2\alpha}}{\Gamma (1-\alpha)}\right)^{k+1}\right)  \left(\frac{0.792 (k+1)}{\log(k+2)}\right)^{k+1}
 \end{align*}
 In particular for $m$ sufficiently large, i.e.  for $m\geq (\Gamma (1-\alpha)/t)^{1/2\alpha}$, the following holds
 $$\bE[(S^\alpha _m(t))^k] \leq  C(\alpha )^kt^km^{k+2\alpha k-3\alpha}\left(\frac{ k+1}{\log(k+2)}\right)^{k+1}$$
 where $C(\alpha )=c(\alpha)0.792/\Gamma (1-\alpha)$.

\end{proof}

Let us now consider the sequence of jump processes $\{W^{N,\beta}_n(t)\}_n$ described in Section \ref{sez-2} and for any couple $(n,m)\in \bN^2$  let us consider the subordinated process $X_{n,m}(t)=W^{N,\beta}_n(S^\alpha _m(t))$.  We can think of $X_{n,m}$ as a jump process in the complex plane, where a random number of jumps, uniformly distributed on the set $\frac{\beta^{1/N}}{n^{1/N}}R(N)$, occur, namely:
\begin{equation}
\label{def-Xnm}
X_{n,m}(t)=\frac{1}{n^{1/N}}\sum_{j=1}^{\lfloor nS^\alpha _m(t)\rfloor}\xi_j,
\end{equation}
where $\xi_j$ are i.i.d.\ uniformly distributed on the set $\beta^{1/N}R(N)$ (see Eq. \eqref{e2}).

The following theorem is the analog of Theorem \ref{teo3} for processes that are driven by subordinators, and shows that $X_{n,m}(t)$ converges in a suitable sense to an $N\alpha $-stable process.

\begin{theorem}\label{teo13}
Let the parameters $N\in \bN$ and $\beta \in \bC$ be chosen in such a way that the following inequality is satisfied 
\begin{equation}\label{cond-bN}{\rm Re }((-i)^N\beta y ^N)\leq 0\quad\forall y\in \bR.
\end{equation}
Then
the following holds
$$\lim_{m\to\infty}\lim_{n\to\infty } \bE[e^{-iy W^{N,\beta}_n(S^\alpha _m(t))}]=e^{- t\left((-1)^{N+1}i^{N}y^N\frac{\beta}{N!}\right)^\alpha} $$
\end{theorem}
    Before  proving theorem \ref{teo13} we give an alternative estimate of the difference between  $\bE[e^{-iy W^{N,\beta}_n(t)}] $ and its limit for $n\to \infty$.
    \begin{lemma}\label{lemmanuovo}
    Under the assumption \eqref{cond-bN}, there exists a constant $C(y)$ depending continuously on the parameter $y$ such that the  following estimate holds
    \begin{align}\label{stimarestonuovo}
|\bE[e^{-iy W^{N,\beta}_n(t)}] -e^{(i)^N\beta t y^N/N!}|
=& \frac{C(y)t}{n}e^{tC(y)}+\frac{1}{n}  \frac{|\beta| |y|^N}{N!}
\end{align}
    \end{lemma}
    \begin{proof} By definition of $W^{N,\beta}_n(t)$=
    $\bE[e^{-iy W^{N,\beta}_n(t)}] =(\psi_{\xi_{N,\beta}}(y/n^{1/N}))^{\lfloor nt\rfloor}$, where $ \psi_{\xi_{N,\beta}}$ is the characteristic function of the complex random variable $\xi_{N,\beta}$ defined in \eqref{e2}. Hence:
   \begin{align}\nonumber
|\bE[e^{-iy W^{N,\beta}_n(t)}] -e^{(i)^N\beta t y^N/N!}|
=&
|(\psi_{\xi_{N,\beta}}(y/n^{1/N}))^{\lfloor nt\rfloor} -e^{\frac{(i)^N\beta }{N!}\frac{\lfloor nt\rfloor}{n} y^N}+e^{\frac{(i)^N\beta }{N!}\frac{\lfloor nt\rfloor}{n} y^N}-e^{(i)^N\beta t y^N/N!}|\\
\leq& |(\psi_{\xi_{N,\beta}}(y/n^{1/N}))^{\lfloor nt\rfloor} -e^{\frac{(i)^N\beta }{N!}\frac{\lfloor nt\rfloor}{n} y^N}|+|e^{\frac{(i)^N\beta }{N!}\frac{\lfloor nt\rfloor}{n} y^N}-e^{(i)^N\beta t y^N/N!}|\label{stimanuova}
\end{align} 
The first term can be estimated as
  \begin{align*}
    |(\psi_{\xi_{N,\beta}}(y/n^{1/N}))^{\lfloor nt\rfloor} -e^{\frac{(i)^N\beta }{N!}\frac{\lfloor nt\rfloor}{n} y^N}|\leq&  |(\psi_{\xi_{N,\beta}}(y/n^{1/N})) -e^{\frac{(i)^N\beta }{N!}\frac{y^N}{n} }|\sum_{j=0}^{\lfloor nt\rfloor -1}|\psi_{\xi_{N,\beta}}(y/n^{1/N})) |^j
    \end{align*} 
    By setting $r(n,\beta, y):=\psi_{\xi_{N,\beta}}(y/n^{1/N}) -e^{\frac{(i)^N\beta }{N!}\frac{y^N}{n} }$, the latter estimate takes the following form:
     \begin{align*}
    |(\psi_{\xi_{N,\beta}}(y/n^{1/N}))^{\lfloor nt\rfloor} -e^{\frac{(i)^N\beta }{N!}\frac{\lfloor nt\rfloor}{n} y^N}|\leq&  |r(n,\beta ,y)|\sum_{j=0}^{\lfloor nt\rfloor -1}|e^{\frac{(i)^N\beta }{N!}\frac{\lfloor nt\rfloor}{n} y^N}+r(n,\beta ,y)|^j\\
    \leq&  |r(n,\beta ,y)|\sum_{j=0}^{\lfloor nt\rfloor -1}(1+|r(n,\beta ,y)|)^j\\
    \leq&  |r(n,\beta ,y)|\lfloor nt\rfloor(1+|r(n,\beta ,y)|)^{\lfloor nt\rfloor}
    \end{align*} 
    Moreover 
      \begin{align*}
    |r(n,\beta ,y)|= &|\psi_{\xi_{N,\beta}}(y/n^{1/N}) -e^{\frac{(i)^N\beta }{N!}\frac{y^N}{n} }|\\
    =&|\frac{\psi^{(N)}_{\xi_{N,\beta}}(0) }{N!}\frac{y^N}{n}+\frac{\psi^{(2N)}_{\xi_{N,\beta}}(\tilde z) }{(2N)!}\frac{y^{2N}}{n^2}-\frac{i^Ny^N\beta}{N! n}-\frac{1}{2}\left(\frac{i^Ny^N\beta}{N! n}\right)^2e^{\frac{(i)^N\beta }{N!}z }|
      \end{align*} 
      where $z,\tilde z\in \bR$, with  $z\in [0, \frac{y^N}{n}]$ and $\tilde z\in [0,y/n^{1/N}]$. By formula \eqref{e3}  and the boundedness of the continuos map $\psi^{(2N)}_{\xi_{N,\beta}}$ over the interval $ [0,y]$, i.e. $|\psi^{(2N)}_{\xi_{N,\beta}}(\tilde z) |\leq M$ $\forall \tilde z \in  [0,y]$, we obtain:
   \begin{align*}
    |r(n,\beta ,y)|\leq & \frac{|\beta|^2|y|^{2N}}{n^2}\left(\frac{1}{(2N)!}|\psi^{(2N)}_{\xi_{N,\beta}}(\tilde z) |+\frac{1}{2(N!)^2}\right)\leq \frac{|\beta|^2|y|^{2N}}{n^2}\left(\frac{M}{(2N)!}+\frac{1}{2(N!)^2}\right)   
         \end{align*} 
         hence there exists a constant $C(y)$ depending continuously on the parameter $y$ such that  $|r(n,\beta ,y)|\leq \frac{C(y)}{n^2}$. We then obtain 
         \begin{align*}
    |(\psi_{\xi_{N,\beta}}(y/n^{1/N}))^{\lfloor nt\rfloor} -e^{\frac{(i)^N\beta }{N!}\frac{\lfloor nt\rfloor}{n} y^N}|\leq& \frac{C(y)}{n^2}\lfloor nt\rfloor\left(1+\frac{C(y)}{n^2}\right)^{\lfloor nt\rfloor}\leq \frac{C(y)t}{n}e^{tC(y)}
          \end{align*}

Further, the second term in \eqref{stimanuova} can be estimated as:
  \begin{align*}
|e^{\frac{(i)^N\beta }{N!}\frac{\lfloor nt\rfloor}{n} y^N}-e^{(i)^N\beta t y^N/N!}|= &|e^{\frac{(i)^N\beta }{N!}\frac{\lfloor nt\rfloor-nt}{n} y^N}-1| 
\leq \frac{1}{n}  \frac{|\beta| |y|^N}{N!}
 \end{align*} 
 and we eventually obtain 
  \begin{align*}
|\bE[e^{-iy W^{N,\beta}_n(t)}] -e^{(i)^N\beta t y^N/N!}|
=& \frac{C(y)t}{n}e^{tC(y)}+\frac{1}{n}  \frac{|\beta| |y|^N}{N!}
\end{align*} 
    \end{proof}
\begin{proof}(of theorem \ref{teo13})\\
By lemma \ref{lemmanuovo}
\begin{align*}
\bE \left[ e^{-i y W^{N,\beta}_n(S^\alpha _m(t))} \right]
=& \bE \left[ e^{(i)^N\beta S^\alpha _m(t) y^N/N!}\right]+ R(n,m,y,t)
\end{align*} 
where $| R(n,m,y,t)|\leq  \bE \left[ \frac{C(y)S^\alpha _m(t)}{n}e^{S^\alpha _m(t)C(y)}+\frac{1}{n}  \frac{|\beta| |y|^N}{N!}\right]$.\\
Since, by estimate \eqref{eq:st_mom_comp-2}, the expectation $\bE \left[ S^\alpha _m(t)e^{S^\alpha _m(t)C(y)}\right]$ is finite for any $m\in \bN$, we obtain
\begin{align*}
\lim_{n\to\infty}\bE \left[ e^{-i y W^{N,\beta}_n(S^\alpha _m(t))} \right]
=& \bE \left[ e^{(i)^N\beta S^\alpha _m(t) y^N/N!}\right]=
\end{align*} 
Eventually, by using  \eqref{eq:approx_poisson_comp-gen}, we get
\begin{align*}
\lim_{m \to \infty} \lim_{n \to \infty} \bE \left[ e^{-i y W^{N,\beta}_n(S^\alpha _m(t))} \right]=\lim_{m \to \infty}\bE \left[\exp\left( (-i)^{N} y^{N} \frac{\beta S^\alpha _m(t)}{N!}  \right)\right]
= \exp\left(-  t \left((-1)^{N+1} i^{N} y^{N} \frac{\beta}{N!} \right)^\alpha \right) \,  
\end{align*}
\end{proof}
\begin{theorem}\label{teo14}
Let $f:\bC\to \bC$ be the Fourier transform of a complex bounded measure $\mu$ on $\bR$ with compact support, i.e. $f$ be of the form $f(x)=\int e^{-iy x}d\mu(y)$. Then under the above assumptions on $\beta $ and $N$
\begin{equation}\label{lim-f}
 \lim_{m\to\infty}\lim_{n\to\infty } \bE[f(W^{N,\beta}_n(S^\alpha _m(t)))] = \int_\bR e^{- t\left((-1)^{N+1}i^{N}y^N\frac{\beta}{N!}\right)^\alpha} \, {\rm d}\mu (y)
\end{equation}
\end{theorem}
\begin{proof}
Let $K\subset  \bR$ be the support of the measure $\mu$. 
By lemma \ref{lemmanuovo}, there exists a positive constant $M\in \bR^+$ such that for any $y\in K$ the following holds:
\begin{align*} \bE \left[ e^{-i y W^{N,\beta}_n(S^\alpha _m(t))} \right]
=& \bE \left[ e^{(i)^N\beta S^\alpha _m(t) y^N/N!}\right]+ R(n,m,,t)
\end{align*} 
where $| R(n,m,t)|\leq  \bE \left[ \frac{MS^\alpha _m(t)}{n}e^{S^\alpha _m(t)M}+\frac{1}{n}  \frac{|\beta| |y|^N}{N!}\right]<+\infty$.\\
By Fubini theorem
\begin{align*}
\bE[f(W^{N,\beta}_n(S^\alpha _m(t))] =&\int_\bR \bE\left[ e^{-i y W^{N,\beta}_n(S^\alpha _m(t))} \right] \, {\rm d}\mu(y)\\
=& \int_\bR \bE \left[\exp\left( (-i)^{N} y^{N} \frac{\beta S^\alpha _m(t)}{N!}  \right)\right] \, {\rm d}\mu (y)+|\mu|R(n,m,t),
\end{align*}
where $|\mu| $ denotes the total variation of the complex measure $\mu$.
By letting $n\to\infty$ we obtain:
\begin{align*}
\lim_{n\to\infty}
\bE[f(W^{N,\beta}_n(S^\alpha _m(t))] =&\int_\bR \bE \left[\exp\left( (-i)^{N} y^{N} \frac{\beta S^\alpha _m(t)}{N!}  \right)\right]\, {\rm d}\mu(y)\\
\end{align*}
Eventually, by dominated convergence theorem, the following holds
\begin{align*}
\lim_{m\to\infty}
\lim_{n\to\infty}
\bE[f(W^{N,\beta}_n(S^\alpha _m(t))] =&\lim_{m\to\infty}\int_\bR \bE \left[\exp\left( (-i)^{N} y^{N} \frac{\beta S^\alpha _m(t)}{N!}  \right)\right] \, {\rm d}\mu(y)\\
=&\int_\bR \lim_{m\to\infty}\bE \left[\exp\left( (-i)^{N} y^{N} \frac{\beta S^\alpha _m(t)}{N!}  \right)\right]\, {\rm d}\mu(y)\\
=&\int_\bR e^{- t\left((-1)^{N+1}i^{N}y^N\frac{\beta}{N!}\right)^\alpha} \, {\rm d}\mu (y)
\end{align*}

\end{proof}

\begin{remark}\label{rem-Nalphaprocess}
Theorem \ref{teo13} allows to interpret {\em formally} the limit process of $X_{n,m}$ as an $N\alpha$-stable process. 
In fact such a process cannot exist in the case $N\alpha>2$ and, analogously to the random walk $W^{N,\beta}_n$ studied in section \ref{sez-2}, the sequence of complex random variables $X_{n,m}(t)$ does not converge in distribution. 
Theorems \ref{teo13} and \ref{teo14} have to be interpreted in a weaker sense, indeed even if the distribution of $W^{N,\beta}_n(S^\alpha _m(t))$ does not converge to a well defined probability measure on the complex plane, the integral of suitable functions (i.e. linear combinations of exponentials) converges and the limit is given by formula \eqref{lim-f}.
\\
It is particularly interesting the study of the case  $N$ being an integer strictly greater than 2  and the product $N\alpha$ satisfies the inequality $N\alpha\leq 2$.  
In this case an $N\alpha $-stable process $H^{N\alpha}(t)$ exists and its relation with the sequence of jump processes $\{W^{N,\beta}_n(S^\alpha _m(t))\}_{m,n}$ is worth of investigation. 
We can consider for instance the case where $N=2M$ with $M\in \bN$, $\alpha =1/M$ and   $\beta= (-1)^{M+1}\frac{(2M)!}{2^M}$.  
According to Theorem \ref{teo13} we have the following convergence result 
$$\lim_{m\to\infty}\lim_{n\to\infty } \bE[e^{-iy W^{N,\beta}_n(S^\alpha _m(t))}]=e^{- t\frac{y^2}{2}} , \qquad y\in \bR.$$
Nevertheless this is not sufficient to interpret the limit of $W^{N,\beta}_n(S^\alpha _m(t))$ as a {\color{black}real valued} Wiener process. 

Indeed,  for any $(n,m)\in \bN^2$ the process $W^{N,\beta}_n(S^\alpha _m(t))$ has complex paths and we can prove that
for any $t > 0$ given, the law of the random variable $W_n^{N,\beta}(S_m^\alpha(t))$
cannot converge to a  Gaussian distribution on the real axis.
Actually, a straightforward computation shows that
\begin{align*}
\lim_{n,m\to\infty} \bP\left( W_n^{N,\beta}(S_m^\alpha(t))\in B_R(0)\right)=0, 
\end{align*}
for any $R > 0$ given and $B_R(0)\subset\bC$.
Therefore, 
\begin{align*}
\lim_{n,m\to\infty} \bP\left( \left\{ \big|{\rm Re}[W_n^{N,\beta}(S_m^\alpha(t))]\big| \le \sqrt{2} R \right\} \, \cap \, \left\{ \big|{\rm Im}[W_n^{N,\beta}(S_m^\alpha(t))]\big| \le \sqrt{2} R \right\} \right)=0, 
\end{align*}
hence, even in the case where the imaginary part disappears, the real part cannot have a Gaussian distribution since it is concentrated outside the interval $(-\sqrt{2}R, \sqrt{2}R)$.

\end{remark}

\section{Probabilistic representation of evolution equations with fractional order space derivative}
\label{sez4}

Let $A_{N,\beta}:D(A_{N,\beta})\subset L^2(\bR)\to L^2(\bR)$  be the operator defined by 
$$\widehat{A_{N,\beta} f}(y):=\frac{(-i)^N\beta y^N}{N!}\hat f(y)$$
where $\hat f$ is the Fourier transform of $f\in L^2(\bR)$, i.e. $\hat f(y)=\int_\bR e^{ixy}f(x)dx$ for $f\in L^1(\bR)$, with   $$D(A_{N,\beta})=\{ f\in  L^2(\bR): \int_\bR y ^{2N}|\hat f(y)|^2 \, {\rm d}y<\infty\}.$$
In other words $A_{N,\beta }$ is the Fourier integral operator with symbol $\Psi(y):=\frac{(-i)^N\beta y^N}{N!}$. On smooth  functions $f\in L^2(\bR)\cap C^N(\bR)$ it is given by 
$$A_{N,\beta }f(x)=\frac{\beta}{N!}\frac{\partial^N}{\partial x^N} f(x).$$
{\color{black}
In the following we shall always assume that $\beta\in \bR$ is a real constant  such that, whenever $N\in \bN$ is even, then  the inequality ${\rm Re}((-1)^{N/2}\beta) \le 0$ is satisfied. 
Under this assumption the operator $A_{N,\beta}$ generates a strongly continuous contraction semi-group on $L^2(\bR)$. 
In addition, for $N\in \bN$ odd, the operator $iA_{N,\beta}$ is self-adjoint and generates  a strongly continuous unitary group  on $L^2(\bR)$.
}

Let $B :D(B)\subset L^2(\bR)\to L^2(\bR)$  be the operator defined by 
$$\widehat{B f}(y):=|y|\hat f(y)$$
with   $$D(B)=\{ f\in  L^2(\bR): \int_\bR y ^{2}|\hat f(y)|^2 \, {\rm d}y<\infty\}$$
$B$ is called {\em Riesz operator}, formally written as $B\equiv \partial _{|x|}$.  It is given by 
$$\partial _{|x|}f(x)=-k\int_0^\infty \frac{f(x-s)-2f(x)+f(x+s) }{s^2} \, {\rm d}s, \qquad x\in \bR, \qquad k=\left(2\int_0^\infty \frac{1-\cos s}{s^2} \, {\rm d}s\right)^{-1}=\frac{1}{\pi}$$
Via functional calculus, it is straightforward to define the $N$-power of $B$, as the operator $B^N$ with symbol $\Psi(y)=|y|^N$ and domain $D(B^N)=\{f\in L^2(\bR): \int _\bR y^{2N}|\hat f (y)|^2 \, {\rm d}y<\infty\}$.  For $N$  even we have that $A_{N,\beta}=\frac{(-i)^N\beta}{N!} B^N$.\\
For a given real constant $\alpha \in (0,1)$, let us define the fractional power of  $-A_{N,\beta}$ and $B^N$ as the operators $(-A_{N,\beta})^\alpha$ and $B_{N,\alpha}$  with symbols respectively 
\begin{eqnarray*}
\widehat{(-A_{N,\beta})^\alpha f} (y)&=&\left(\frac{(-1)^{N+1}i^N\beta y^N}{N!}\right)^\alpha\hat f(y),\\
\widehat{B_{N,\alpha} f} (y)&=&|y|^{N\alpha}\hat f(y),\\
\end{eqnarray*}
where, given a complex number $z\in \bC$, with $z=|z|e^{i\theta}$, with $\theta \in (-\pi, \pi]$, the $\alpha $-power of $z$ is taken as $$z^\alpha =|z|^\alpha e^{i\alpha\theta}.$$
Note that when $N$ is odd and $i^N\beta$ is a purely  imaginary number the symbol of $(-A_{N,\beta})^\alpha$ is explicitly given by $\left(\frac{(-1)^{N+1}i^N\beta y^N}{N!}\right)^\alpha:=\frac{|\beta|^\alpha}{N!^\alpha}|y |^{N\alpha}e^{ (-1)^{\frac{N-1}{2}}\frac{i\pi \alpha}{2}\frac{\beta y}{|\beta y|}}$.\\
The action of the operator $(-A_{N,\beta})^\alpha$ can be also represented by the following formula
\begin{align*}
(-A_{N,\beta})^\alpha f(x)&= \frac{\alpha}{\Gamma(1-\alpha)}\int_0^\infty \frac{f(x)-e^{s{A_N,\beta}}f(x)}{s^{\alpha +1}} \, {\rm d}s
\\
&= \frac{\alpha}{\Gamma(1-\alpha)}\int_0^\infty \lim_{n\to\infty}\frac{f(x)-\bE[f(x+W_n^{N, \beta}(s))]}{s^{\alpha +1}} \, {\rm d}s
\end{align*}
$e^{s{A_N,\beta}}$ being the semigroup generated by $A_{N,\beta}$ and $W_n^{N,\beta}$ the sequence of complex random walks defined in section \ref{sez-2}.\\
 For $N$ an even integer, there is a trivial relation between the  operators $(-A_{N,\beta})^\alpha $ and $B_{N,\alpha}$, namely $$(-A_{N,\beta})^\alpha=\left(\frac{(-1)^{N+1}i^N\beta }{N!}\right)^\alpha B_{N,\alpha}.$$
For $N$ odd,  we have 
$B_{N,\alpha}=\frac{(A_{N,\beta})^\alpha+(-A_{N,\beta})^\alpha}{2(|\beta|/N!)^\alpha\cos(\pi\alpha /2)}$. 
The action of $B_{N,\alpha}$  can be also written in the following form:
\begin{align*}
B_{N,\alpha} f(x)&= \left(2\int_0^\infty \frac{1-\cos s}{s^{\alpha +1}} \, {\rm d}s\right)^{-1}\int_0^\infty \frac{e^{sA_{N,1}}f(x)-2f(x)+e^{-sA_{N,1}}f(x)}{s^{\alpha+1}} \, {\rm d}s
\\
&= \left(2\int_0^\infty \frac{1-\cos s}{s^{\alpha +1}} \, {\rm d}s\right)^{-1}\int_0^\infty \lim_{n\to \infty}\frac{\bE[f(x+W^{N, 1}_n(s))]-2f(x)+\bE[f(x+W^{N, -1}_n(s))]}{s^{\alpha+1}} \, {\rm d}s.
\end{align*}

Let $f\in  L^2(\bR)$  be a function of the form 
\begin{equation}\label{as-init-datum}
f(x)=\frac{1}{2\pi}\int_\bR e^{-ixy}\hat f(y) \, {\rm d}y, \qquad x\in \bR,
\end{equation}
with $\hat f\in L^2(\bR)$ being a compactly supported function.  It is straightforward to verify that $f$ belongs to the domain of any of the operators above. Furthermore the subset $D\subset L^2(\bR)$ of functions of the form \eqref{as-init-datum} is an operator core. 
By considering the complex bounded Borel measure on the real line $\mu_f$ absolutely continuous with respect to the Lebesgue measure with density $ \frac{\hat f}{2\pi}$, we can look at the function $f$ defined by \eqref{as-init-datum} as the Fourier transform of $\mu_f$. 
Hence $f$ can be extended to an entire analytic function $f:\bC\to \bC$ of exponential type (see Remark \ref{lemmacrescita}).
\par
For any of the following  initial value problems
 \begin{eqnarray}
\partial_ t u(t)&=&Au(t)\nonumber\\
u(t_0)&=&f,\qquad t\geq t_0,\ f\in D\label{eqAgen}
\end{eqnarray}
with $A$ being either one of the operators $A_{N,\beta}$ and $-B^N$ or one of their fractional powers $-(-A_{N,\beta})^\alpha $ and $-B_{N,\alpha}$ and with $f$ of the form \eqref{as-init-datum}, we are going to construct a sequence of complex jump processes $\{X_{n,m}\}_{n,m\in \bN}$ providing a probabilistic representation for the solution $u(t,x)$ of the form
$$u(t,x)=\lim_{m\to \infty}\lim_{n\to \infty}\bE[f(x+X_{n,m}(t-t_0))].$$

The first result is taken from \cite{BoMa15} and is a direct consequence of Remark \ref{lemmacrescita}.
\begin{theorem}\label{teo-rapANb}
Let $f:\bR\to \bC$ be an $L^2(\bR)$ function of the form \eqref{as-init-datum}. Then the (classical \footnote{For a classical solution of the initial value problem \eqref{eqANb} we mean a function $u:[0,+\infty)\times \bR\to \bC$ which is of class $C^1$ in the time variable $t$, of class $C^N$ in the space variable $x$ and such that for any $(t,x)\in [0,+\infty)\times \bR$ the equality $\partial_ t u(t,x)=\frac{\beta}{N!}\frac{\partial^N}{\partial x^N} u(t,x)$ holds. }) solution of
 \begin{eqnarray}
\partial_ t u(t,x)&=&A_{N,\beta}u(t,x)\nonumber\\
u(t_0,x)&=&f(x),\qquad t\geq t_0, x\in \bR\label{eqANb}
\end{eqnarray}
 is given by 
 \begin{equation}
 \label{rap-ANb}
 u(t,x)=\lim_{n\to \infty}\bE[f(x+W_n^{N,\beta}(t-t_0))], 
 \end{equation}
 where $\{W_n^{N,\beta}(t)\}_{n\in \bN}$ is the sequence of complex random walks defined in \eqref{Wn}.
\end{theorem}

\begin{proof}
By formula \eqref{eq:cor-6}  we have 
\begin{align*}
u(t,x)=\lim_{n\to \infty}\bE[f(x+W_n^{N,\beta}(t-t_0))] = \frac{1}{2\pi}\int_\bR e^{-ixy} \exp\left((-i)^N\beta (t -t_0)\frac{y^N}{N!}\right)\hat f(y) \, {\rm d}y.
\end{align*}
By the compactness of the support of the function $\hat f$, the function $u(t,x)$ is smooth and, by direct computation, a classical solution of the PDE
\begin{align*}
\partial_ t u(t,x)=\frac{\beta}{N!}\partial ^N_x u(t,x).
\end{align*}
\end{proof}

Let us consider now the fractional power  $(-A_{N,\beta})^\alpha$ of the $N$-order differential operator $-A_{N,\beta}$ and construct a probabilistic representation of the associated $C_0$-contraction semigroup.

\begin{theorem}\label{th10}
Let $f:\bR\to \bC$ be an $L^2(\bR)$ function of the form \eqref{as-init-datum}. Then the  solution of
\begin{eqnarray}
\partial_ t u(t,x)&=&-(-A_{N,\beta})^\alpha u(t,x)\nonumber\\
u(t_0,x)&=&f(x),\qquad t\geq t_0, x\in \bR\label{eqANa}
\end{eqnarray}
 is given by 
 \begin{equation}\label{rep-aNa}u(t,x)=\lim_{m\to \infty}\lim_{n\to \infty}\bE[f(x+X_{n,m}(t-t_0))], \end{equation}
 where $\{X_{n,m}(t)\}_{n,m\in \bN}$ is the sequence of complex random walks defined  
  in \eqref{def-Xnm} as $X_{n,m}(t)=W_n^{N,\beta}(S_m^\alpha (t))$.
\end{theorem}

Equation \eqref{eqANa} can be formally written as
$$
\frac{\partial}{\partial t}u(t,x)=-\left(-\frac{\beta}{N!}\right)^\alpha \frac{\partial^{N\alpha}}{\partial x^{N\alpha}}u(t,x).$$
\begin{proof}
By theorem \ref{teo14} the function $u(t,x) $ defined by the r.h.s. of \eqref{rep-aNa} is equal to 
\begin{align*}
u(t,x) = \lim_{m\to \infty}\lim_{n\to \infty}\bE[f(x+X_{n,m}(t-t_0))]
 = \frac{1}{2\pi}\int_\bR e^{-ixy}e^{- (t-t_0)\left((-1)^{N+1}i^{N}y^N\frac{\beta}{N!}\right)^\alpha}\hat f (y) \, {\rm d}y.
\end{align*}
The last line is exactly the action of the semigroup $e^{-(-A_{N,\beta})^\alpha(t-t_0)}$ on the vector $f\in L^2(\bR)$. Moreover, because of the assumptions \eqref{cond-bN} on the constants $N,\beta$ and the compactness of the support  of $\hat  f\in L^2(\bR)$, the function $u(t,x)$ is smooth in both the time and space variables.
\end{proof}
Let us now consider the Riesz operator $B$ and its powers.  Given $N\in \bN$, let us consider the initial value problem
\begin{eqnarray}
\partial_ t u(t,x)&=&-B^Nu(t,x)\nonumber\\
u(t_0,x)&=&f(x),\qquad t\geq t_0, x\in \bR\label{eq-BN}
\end{eqnarray}
with $f$ of the form \eqref{as-init-datum} as above. For $N$ even, the operator $B^N$ coincides with $(-A_{N,\beta})$, with $\beta\equiv (-1)^{\frac{N}{2}+1} N!$. By Theorem \ref{teo-rapANb} the solution of \eqref{eq-BN} is given by \eqref{rap-ANb}. In the case where $N$ is odd,  the construction of the associated process is neither simple, nor unique, as the following result shows.

\begin{theorem}
Let us consider problem \eqref{eq-BN} with $N \in \bN$ and $f\in D$.
For any $M \in \bN$, let us consider  the sequence of processes $\{W^{2MN,\beta}_n (t)\}$ defined as in \eqref{Wn}, 
with  $\beta = (-1)^{MN+1}(2MN)!$. 
Then, choosing $\alpha =\frac{1}{2M}$, according to Theorem \ref{th10} the solution of the initial value problem \eqref{eq-BN} is given by 
\begin{equation}\label{rap-BN}
u(t,x)=\lim_{m\to\infty }\lim_{n\to\infty }\bE[f(x+W^{2MN,\beta}_n(S^{\frac{1}{2M}} _m(t-t_0)))],
\end{equation}
where $S^{\alpha} _m(t)$ is the sequence of processes \eqref{Sm} approximating the $\alpha$-stable subordinator.
\end {theorem}

\begin{proof}
By an application of Theorem \ref{teo14}, the function $u(t,x) $ defined by the r.h.s. of \eqref{rap-BN} is equal to
$$u(t,x)=\frac{1}{2\pi}\int_\bR e^{-ixy} e^{-(t-t_0)|y|^N}\hat f(y) \, {\rm d}y.$$
\end{proof}

%
%
%
%

More generally, let us consider the operator $B_{N,\alpha}$, and the corresponding associated initial value problem
\begin{eqnarray}
\partial _t u(t,x)&=&-B_{N,\alpha}u(t,x)\nonumber\\
u(t_0,x)&=&f(x)\label{eq-BNa}
\end{eqnarray}
with $f\in L^2(\bR)$ of the form \eqref{as-init-datum}.
For $N$ even, we can take the process $W^{N,\beta}_n(t)$ associated to the operator $A_{N,\beta}$, with $\beta=(-1)^{N/2}N!$ and the solution of \eqref{eq-BNa} is given by
$$u(t,x)=\lim_{m\to\infty }\lim_{n\to\infty }\bE[f(x+W^{N,\beta}_n(S^\alpha _m(t-t_0)))].$$

The following construction allows to handle the  case  where $N$ is odd.
Indeed, recall from \cite{BoMa15} that
the distribution of $-\xi_{N,\beta}$ is equal to the distribution of $\xi_{N,-\beta}$ and the same holds for the corresponding continuous time processes $W^{N,\beta}_n(t)$.

\begin{theorem} 
Let  $W^{N,\beta}_n(t)$ and $\tilde W^{N,\beta'}_n(t)$  be   two independent copies of the process \eqref{Wn}, with $\beta=N!$ and $\beta' =-N!$ respectively and  let $S^\alpha _m(t)$ and $\tilde S^\alpha _m(t)$be two independent copies of the process \eqref{Sm}. Taking a rescaled time variable $\tilde t :=(2\cos\left(\frac{\alpha \pi}{2}\right))^{-1}t$, the solution of \eqref{eq-BNa} with $f\in D$ is given by
$$u(t,x)=\lim_{m\to \infty}\lim_{n\to \infty}\bE\left[ f\left( x+W^{N,\beta}_n \left( S^\alpha _m(\tilde t-\tilde t_0)\right)+\tilde W^{N,\beta'}_n\left( \tilde S^\alpha _m(\tilde t-\tilde t_0)\right) \right)\right]$$
\end{theorem}
\begin{proof}
By applying Theorem \ref{teo14} we obtain
\begin{align*}
u(t,x) = \frac{1}{2\pi}\int_\bR e^{-ixy} e^{- (\tilde t-\tilde t_0)\left(( i y^N)^\alpha +(-i y^N)^\alpha\right)}\hat f(y) \, {\rm d}y
 = \frac{1}{2\pi}\int_\bR e^{-ixy} e^{-(t-t_0)|y|^{N\alpha}}\hat f(y) \, {\rm d}y.
\end{align*}
\end{proof}
\begin{remark}
In fact for any operator of the form $B^N$, with $N\in \bN$, there exists a family of associated processes. For instance, if $N=2$, the solution of the heat equation
\begin{equation}\label{heat-rem8}\partial _t u(t,x)=\frac{1}{2}\partial ^2_x u(t,x)\end{equation} can be represented by means of the Feynman-Kac formula, as the expectation with respect to the distribution of the Wiener process $W(t)$:
$$ u(t,x)=\bE [u(t_0,x+W(t-t_0))],\qquad t\geq t_0, x\in \bR $$
but alternative constructions are possible. Let us consider  for instance the  sequence of processes $W^{N,\beta}_n(S^\alpha_m(t))$, with $N=4$, $\alpha =1/2$ and $\beta =-3!$. Then the solution of \eqref{heat-rem8} can also be represented by 
$$ u(t,x)=\lim_{m\to\infty}\lim_{n\to \infty}\bE [u(t_0,x+W_n^{N,\beta}(S_m^\alpha(t-t_0)))], \qquad t\geq t_0,\ x\in \bR. $$
The same formula holds by taking a generic $M\in \bN$ and setting $N\equiv 2M$, $\alpha\equiv 1/M$ and $\beta\equiv (-1)^{M+1}\frac{(2M)!}{2^M}$
It is important to remark that, even in these cases the sequence of jump processes $\{W_n^{N,\beta}(S_m^\alpha(t))\}_{m,n}$ does not converge to the Wiener process $W$, as discussed in Remark \ref{rem-Nalphaprocess}.
\end{remark}


\section{Time fractional diffusion equations}
\label{sez6}

In this section we consider time-fractional equations of the form
\begin{equation}
\label{eq-frac-t}
\begin{aligned}
\dtime^\alpha_t u(t,x) = & A_{N,\beta}u(t,x) \\
u(0,x)= & f(x),
\end{aligned}
\end{equation}
where the time-fractional derivative $\dtime_t^\alpha$ must be understood in the sense of Caputo%
{\color{black}. Since throughout the paper $\alpha \in (0,1)$, we can define
\begin{align*}
\dtime^\alpha_t v(t) = \frac{1}{\Gamma(1-\alpha)} \int_0^t (t - s)^{-\alpha} \frac{\partial}{\partial s}v(s) \, {\rm d}s.
\end{align*}
The Caputo derivative $\dtime^\alpha_t v(t)$, for $\alpha \in (0,1)$, can also be defined as the function with
Laplace transform $\widetilde{\dtime^\alpha v}(\lambda) = \lambda^\alpha \tilde v(\lambda) - \lambda^{\alpha-1} v(0^+)$.
}
The reader can consult the book  by Samko et al. \cite{SKM} for further details.

In our construction, equation \eqref{eq-frac-t} is solved with the aid of a random time change of the complex random walk $W^{N,\beta}_{n}(t)$.
In order to explain our construction, we introduce
the $\alpha $-stable subordinator $H^\alpha(t)$, i.e., a subordinator with zero drift and L\'evy measure
\begin{align*}
m_\alpha({\rm d}x) = \frac{c}{x^{1+\alpha}} \, \mathds 1_{(0,\infty)}(x) \, {\rm d}x
\end{align*}
where $c > 0$ is a given constant. If we choose $c = \frac{1}{\Gamma(1-\alpha)}$, then the Laplace exponent becomes $\Phi(\lambda) = \lambda^\alpha$.
\\
We denote by $L^\alpha(t)$ the inverse of the subordinator $H^\alpha(t)$ (see Appendix A), namely:
\begin{align*}
L^\alpha(t) = \inf\{s \geq 0\,:\, H^\alpha(s) \geq t \}.
\end{align*}
The moments of $L^\alpha_t$ are equal to
\begin{align*}
\bE[(L^\alpha(t))^k]=\frac{k! t^{\alpha k}}{\Gamma(\alpha k+1)}, \qquad k\in \bN,
\end{align*}
while its Laplace transform  is
\begin{align*}
\bE[e^{-zL^\alpha(t)}]=\sum_{k\geq 0}\frac{(-z)^kt^{\alpha k}}{\Gamma(\alpha k +1)}, \qquad z\in \bC.
\end{align*}
In the following result, we show that the sequence of  subordinated processes $W_n^{N,\beta}(L^\alpha_t)$ can be associated with the PDE \eqref{eq-frac-t}.
Notice that, as opposite to the previous section, here we only need to take one limit.

\begin{theorem}
Let $\beta \in \bC$ and $N\in \bN$ satisfying assumption \eqref{cond-bN}.
If $f:\bR\to\bR$ is an analytic function that is the Fourier transform of a bounded complex Borel measure on $\bR$ with compact support, i.e.
\begin{align*}
f(x)=\int_\bR e^{-i\lambda x} \, {\rm d}\mu_f(\lambda), 
\end{align*}
then of the solution of the initial value problem
\begin{equation*}
\tag{\ref{eq-frac-t}}
\begin{aligned}
\dtime^\alpha_t u(t,x) = & A_{N,\beta}u(t,x) \\
u(0,x)= & f(x),
\end{aligned}
\end{equation*}
is given by
\begin{align*}
u(t,x)=\lim_{n\to \infty}\bE \left[f(x+W_n^{N,\beta}(L^\alpha (t))) \right].
\end{align*}
\end{theorem}

\begin{proof} 
By explicit computation 
\begin{align}
u(t,x) &= \lim_{n\to \infty}\bE \left[f(x+W_n^{N,\beta}(L^\alpha (t))) \right] 
= \int_\bR e^{-i\lambda x}\bE[e^{\frac{\beta}{N!}(-i\lambda)^NL^\alpha(t)}] \, {\rm d}\mu_f(\lambda)
\nonumber
\\
&= \int_\bR e^{-i\lambda x}E_\alpha \left( \frac{\beta}{N!}(-i\lambda)^N t^\alpha \right) \, {\rm d}\mu_f(\lambda)
\label{int-M-L}
\end{align}
where $E_\alpha$ denotes the Mittag-Leffler function, namely $E_\alpha(t)=\sum_{k\geq 0}\frac{ t^k}{\Gamma(\alpha k+1)}$.  Under the stated assumption on the  constants $\beta\in\bC$ and $N\in \bN$, the argument $z\equiv \frac{\beta}{N!}(-i\lambda)^N t^\alpha $ is a complex number with non positive real part for any $t\in \bR^+$ and $\lambda \in\bR$, hence the map $\lambda \mapsto E_\alpha \left( \frac{\beta}{N!}(-i\lambda)^N t^\alpha \right)$ is bounded and the integral \eqref{int-M-L} is absolutely convergent and defines a $C^\infty $ function of the $x$ variable. 
Furthermore, for any $t\in \bR^+$ the function $u(t,x)$ is still the Fourier transform of a complex measure $\mu_t$ on $\bR$ with compact support which is also absolutely continuous with respect to $\mu_f$, namely :
$$\frac{d\mu_t}{d\mu_f}:=E_\alpha \left( \frac{\beta}{N!}(-i\lambda)^N t^\alpha \right)$$
In particular
$$A_{N,\beta}u(t,x)= \int_\bR e^{-i\lambda x} \frac{\beta}{N!}(-i\lambda)^N E_\alpha \left( \frac{\beta}{N!}(-i\lambda)^N t^\alpha \right) \, {\rm d}\mu_f(\lambda).$$
In order to prove that this is equal to $\partial^\alpha_t u(t,x)$, hence that equation \eqref{eq-frac-t} holds, let us take the Laplace transform of both sides. Denoting $\tilde u(\rho, x)$ the Laplace transform of $u(t,x)$ and applying Fubini theorem as well as the properties of the Mittag-Leffler function and the condition ${\rm Re}\left( \frac{\beta}{N!}(-i\lambda)^N  \right)\leq 0$,  first of all we have
\begin{align}
\tilde u(\rho,x) &= \int_0^\infty e^{-\rho t}u(t,x) \, {\rm d}t 
= \int_\bR e^{-i\lambda x}\int_0^\infty e^{-\rho t}E_\alpha \left( \frac{\beta}{N!}(-i\lambda)^N t^\alpha \right) \, {\rm d}t \, {\rm d}\mu_f(\lambda)
\nonumber
\\
&= \int_\bR\frac{ e^{-i\lambda x} }{\rho}\frac{1}{1-\left(\frac{\beta}{N!}(-i\lambda)^N \rho^{-\alpha}\right)} \, {\rm d}\mu_f(\lambda)
\label{Laplaceut}
\end{align}
On the other hand by taking the  Fourier-Laplace  transform of both sides of \eqref{eq-frac-t}  we obtain
$$\rho^\alpha \tilde u(\rho, \lambda)-\rho^{\alpha -1} \tilde u(0, \lambda)=\frac{\beta}{N!}(-i\lambda)^N \tilde u(\rho, \lambda)$$
which yields
\begin{equation}\label{Laplaceut-2}
\tilde u(\rho, \lambda)=\frac{\rho^{-1}}{1-\frac{\beta}{N!}(-i\lambda)^N\rho^{-\alpha}}\tilde u(0, \lambda).
\end{equation}
By comparing \eqref{Laplaceut} and \eqref{Laplaceut-2} we obtain the final result.


\end{proof}


\subsection{Time fractional equations and non-local space fractional equations}

In this section, we discuss the relationship between equation \eqref{eq-frac-t} and the diffusion equation with non-local forcing term
of the form
\begin{equation}
\label{eq:m1}
\begin{aligned}
\partial_t u(t,x) &= (A_{N,\beta})^{1/\alpha} u(t,x) + \sum_{k=1}^{1/\alpha-1} \frac{1}{\Gamma(\alpha k)} t^{\alpha k -1} A_{N,\beta}^k f(x)
\\
u(0,x) &= f(x)
\end{aligned}
\end{equation}
We shall require that $\alpha = M^{-1}$ for some $M \in \bN$, $M > 1$, so in particular $\alpha \in (0,1)$.
\\
The proof of the equivalence follows by taking Laplace (in time, parameter $s$) and Fourier (in space, parameter $\lambda$) transform of both
equations, and proving that the Laplace-Fourier transforms of the solutions coincide.

Let us first consider the equation \eqref{eq-frac-t}. 
By taking Laplace and Fourier transform on both sides we get
\begin{align*}
s^\alpha \hat{\tilde u}(s,\lambda) - s^{\alpha-1} \widehat{f}(\lambda) = \symbo \hat{\tilde u}(s,\lambda)
\end{align*}
hence
\begin{align}
\label{eq:m2}
\hat{\tilde u}(s,\lambda) = \frac{s^{\alpha-1}}{s^{\alpha} - \symbo} \widehat{f}(\lambda)
\end{align}
(it is worth noticing that the assumption on the sign of $\beta$: ${\rm Re}((-i)^N \beta) \le 0$, implies that the quantity we simplify on both sides is never zero).
\\
Next, we apply the same machinery to the solution of equation \eqref{eq:m1}. We have
\begin{align*}
s \hat{\tilde u}(s,\lambda) - \widehat{f}(\lambda) = \left(\symbo\right)^M \hat{\tilde u}(s,\lambda)
+ \sum_{k=1}^{M-1} s^{-\alpha k} \left(\symbo\right)^k \widehat{f}(\lambda)
\end{align*}
and rearranging both sides we get
\begin{align*}
\left(s - \left(\symbo\right)^M \right) \hat{\tilde u}(s,\lambda)
&= \sum_{k=1}^{M-1} s^{-\alpha k} \left(\symbo\right)^k \widehat{f}(\lambda)
\\
&= \frac{1 - \left(s^{-\alpha} \symbo\right)^M}{1 - s^{-\alpha} \symbo} \widehat{f}(\lambda)
\\
&= \frac{s - \left(\symbo\right)^M}{s - s^{1-\alpha} \symbo} \widehat{f}(\lambda)
\end{align*}
so we simplify the quantity in the numerator
\begin{align*}
\hat{\tilde u}(s,\lambda)
&=\frac1{s - s^{1-\alpha} \symbo} \widehat{f}(\lambda)
\end{align*}
which coincides with \eqref{eq:m2}, as required.

Informally, we notice that a fractional time derivative of order $\alpha$ has become a fractional space
derivative of order $1/\alpha$, but
this transformation affects, in a rather complicated way, the initial condition.
In the next, and last example of this section, we shall see what happens if we start with an equation
involving fractional
derivatives of the same order in both time and space, and we compare it with an equation of
integer order derivatives. To be precise, the following equation takes place instead of  \eqref{eq-frac-t}
\begin{equation}
\label{eq-frac-t2}
\begin{aligned}
\dtime^\alpha_t u(t,x) = & (A_{N,\beta})^\alpha u(t,x) \\
u(0,x)= & f(x),
\end{aligned}
\end{equation}
while the following nonlocal problem takes the place of \eqref{eq:m1}:
\begin{equation}
\label{eq:m12}
\begin{aligned}
\partial_t u(t,x) &= A_{N,\beta} u(t,x) + \sum_{k=1}^{1/\alpha-1} \frac{1}{\Gamma(\alpha k)} t^{-\alpha k } A_{N,\beta}^{\alpha k} f(x)
\\
u(0,x) &= f(x).
\end{aligned}
\end{equation}
The proof is analog to the previous one. Let us first consider equation \eqref{eq-frac-t2}; we have that
\begin{align*}
s^\alpha \hat{\tilde u}(s,\lambda) - s^{\alpha-1} \widehat{f}(\lambda) = \left(\symbo\right)^\alpha \hat{\tilde u}(s,\lambda)
\end{align*}
so that
\begin{align}
\label{eq:m22}
\hat{\tilde u}(s,\lambda) = \frac{s^{\alpha-1}}{s^{\alpha} - \left(\symbo\right)^\alpha} \widehat{f}(\lambda)
\end{align}
Consider now \eqref{eq:m12}; we have
\begin{align*}
s \hat{\tilde u}(s,\lambda) - \widehat{f}(\lambda) = \left(\symbo\right) \hat{\tilde u}(s,\lambda)
+ \sum_{k=1}^{M-1} s^{-\alpha k} \left(\symbo\right)^{\alpha k} \widehat{f}(\lambda)
\end{align*}
from which we obtain
\begin{align*}
\left(s - \left(\symbo\right) \right) \hat{\tilde u}(s,\lambda)
&= \sum_{k=1}^{M-1} s^{-\alpha k} \left(\symbo\right)^{\alpha k} \widehat{f}(\lambda)
\\
&= \frac{1 - \frac1s \left(\symbo\right)}{1 - \frac{1}{s^\alpha} \left(\symbo\right)^\alpha} \widehat{f}(\lambda)
\end{align*}
which, compared with \eqref{eq:m22}, implies that the solution of \eqref{eq:m12} coincides with that of \eqref{eq-frac-t2}, as required.

\appendix
\section{Fractional derivatives and Bochner's subordination}\label{app-A}

\subsection*{Fractional derivative}
Let $\alpha \in (0,1)$ be a real constant and let us consider the  L\'{e}vy measure $M$ on the positive half line defined by
\begin{equation}
M({\rm d}s)=\frac{\alpha}{\Gamma(1-\alpha)} \frac{{\rm d}s}{s^{\alpha +1}}
\end{equation} 
It is well known that $M$ is the L\'{e}vy measure of an $\alpha$-stable subordinator (\cite{Ber97}), that is a totally (positively) skewed stable process for which the L\'{e}vy-Khinchine formula is written in terms of the Bernstein function \begin{equation}
x^\alpha = \int_0^\infty (1- e^{-xs}) \, M({\rm d}s) .\label{Bernst-rep}
\end{equation}
Formula \eqref{Bernst-rep} is valid for any $x \in \bC$, with ${\rm Re} (x)\geq 0$, in particular for $x =|x|e^{i\theta}$, $\theta \in [-\pi/2, \pi/2]$ it gives $|x|^\alpha e^{i\alpha\theta}= \int_0^\infty (1- e^{-xs})M(ds) $. The representation \eqref{Bernst-rep} is therefore associated with the symbol of a positively skewed stable process, say $H^\alpha(t)$, $t\geq 0$. Indeed, for $\lambda>0$, we have that $\mathbb{E} \left[e^{-\lambda H^\alpha (t)} \right]= e^{-t \lambda^\alpha}$ and we say that $H^\alpha (t)$ is a stable subordinator of order $\alpha \in (0,1)$. It has non-negative increments and therefore, non-decreasing paths. Thus, $H^\alpha(t)$ can be considered as a time change and, given a stochastic process  $X(t)$, one can consider the subordinated process $X (H^\alpha (t))$. We recall that, for $\alpha \uparrow 1$, $H^\alpha (t)$ becomes the elementary subordinator $t$. The density law $h=h(t,x)$ of $H^\alpha(t)$ solves the problem
\begin{equation}
\label{problem-h}
\left\lbrace \begin{array}{l}
\displaystyle \partial_t h = - \partial_x^\alpha h\\
\displaystyle h(0,x)=\delta(x), \quad x \in \mathbb{R}^+\\
\displaystyle h(t,0) = 0, \quad t \in \mathbb{R}^+
\end{array} \right.
\end{equation}
where $\partial_x^\alpha = \partial^\alpha / \partial x^\alpha$ is the Riemann-Liouville fractional derivative with symbol $(-i y)^\alpha = | y |^\alpha e^{-i \frac{\pi \alpha}{2} \frac{y}{|y |}}$: 
$$
\frac{d^\alpha f}{d x^\alpha}(x) :=\frac{1}{2\pi}\int e^{-ixy}(-i y)^\alpha \hat f (y) \, {\rm d}y, \qquad \hat f (y)= \int e^{ixy}f(x) \, {\rm d}x.
$$ 
According with \eqref{Bernst-rep}, 
\begin{equation}
\label{frac-der-H}
\frac{d^\alpha f}{d x^\alpha}(x) = \int_0^\infty (f(x) - f(x-s)) \, M({\rm d}s). 
\end{equation}

We also introduce the inverse to a stable subordinator, that is the non-Markovian process
\begin{align*}
L^\alpha(t) = \inf\{s\geq 0\,:\, H^\alpha(s) \geq t\}, \quad t\geq 0.
\end{align*}
We have that $\mathbb{P}(L^\alpha(t) < x) = \mathbb{P}(H^\alpha(x) > t)$ and $\mathbb{E} \left[\exp -\lambda L^\alpha(t) \right] = E_\alpha(- \lambda t^\alpha)$ with $\lambda\geq 0$ where
\begin{equation}
E_{\beta}(z) := \frac{1}{2\pi i}\int_{Ha} \frac{\zeta^{\beta -1} e^\zeta}{\zeta^\beta -z} \, {\rm d}\zeta =  \sum_{k=0}^\infty \frac{z^k}{\Gamma(\beta k +1)}, \quad {\rm Re}\left(\beta \right) >0, \; z \in \mathbb{C} \label{mittag-leffler}
\end{equation}
($Ha$ is the Hankel path) is the Mittag-Leffler function. We observe that, for $\alpha \in (0,1)$,  $u(t)=E_\alpha(w t^\alpha)$ with $t>0$, $w>0$, is the fundamental solution to the fractional relaxation equation
\begin{equation}
\mathbf{D}^\alpha_t u(t) -w \, u(t)=0\label{relaxEq}
\end{equation}
where 
\begin{equation}
\mathbf{D}^\alpha_t u(t) : = \frac{1}{\Gamma(1-\alpha)} \int_{0}^{t} \frac{\partial_s u(s)\, {\rm d}s}{(t-s)^{\alpha}}
\end{equation}
is the so-called Caputo derivative or the Dzerbayshan-Caputo fractional derivative. 


\subsection*{Fractional power of a generator}
Given   a strongly continuous contraction semigroup $T(t)$ on $(C^\infty(\mathbb{R}), \|\cdot \|_\infty)$ with infinitesimal  generator $(A,D(A))$, we write $T(t)=e^{tA}$, $t>0$. Let us consider the Markov process $X=(\{X(t)\}_{t\geq 0}; \mathbb{P}_x, x \in \mathbb{R})$ with $$T(t)f(x) = \mathbb{E}_xf(X(t)), \qquad f \in D(A),$$ where $\mathbb{E}_x \mathbf{1}_\mathcal{A}(X(t))= \mathbb{P}_x(X(t) \in \mathcal{A})$. We can define the subordinated semigroup given by the Bochner integral
\begin{equation}
\label{subord-semig}
T^\alpha(t) f(x) : = \int_0^\infty T(s) f(x)\, h(t,s) \, {\rm d}s = \mathbb{E}_x f(X(H^\alpha (t))), \quad x\in \mathbb{R},\, t \in \mathbb{R}^+
\end{equation}
where $h(t,s)=\mathbb{P}(H^\alpha (t) \in ds)/ds$ has been introduced before. According to the representation given by Phillips \cite{Phill52}, for $f \in D(A)$, we also define
\begin{equation}
\label{power-A-Def}
-(-A)^\alpha f(x) := \int_0^\infty (T(s)f(x) - f(x)) \, M({\rm d}s).
\end{equation}
The formal representation $e^{tA}$ of $T(t)$ shows that $A^\alpha = -(-A)^\alpha$ given in \eqref{power-A-Def} is the generator of $T^\alpha$ given in $\eqref{subord-semig}$ by considering a functional calculus which is referred to as Bochner-Phillips calculus (see for example \cite{BergBoyDelau93}). The special case we introduced here can be extended to  general time-changed processes with infinitesimal generators $-\phi(-A)$ where $\phi: (0, \infty) \mapsto [0, \infty)$ is a Bernstein function, that is the symbol of a time change (a non decreasing process). Then, $-\phi(-A)$ is characterized via resolvents in terms of Dunford-Taylor integrals (\cite{BergBoyDelau93, Shill96}). For the generators we considered so far we evidently have that the corresponding (Fourier) symbols are written in terms of Bernstein functions. If $-\Psi$ is the symbol of the generator $A$, i.e. for $f$ belonging to the Schwartz space  $\mathcal{S}(\mathbb{R})$  of rapidly decaying $C^\infty(\mathbb{R})$ functions, 
$$
A f(x) = -\frac{1}{2\pi} \int_{-\infty}^{+\infty} e^{-i \lambda x} \Psi(\lambda) \widehat{f}(\lambda) \, {\rm d}\lambda
$$
then the fractional power of $-A$ can be represented as:
$$
(-A)^\alpha f(x) = \frac{1}{2\pi} \int_{-\infty}^{+\infty} e^{-i \lambda x} (\Psi(\lambda))^\alpha \widehat{f}(\lambda) \, {\rm d}\lambda.
$$

\section{Some estimates}\label{app-B}
\subsection{Moments of the Poisson distribution}
Let $X$ be a Poisson random variable with parameter $\lambda$.
Then the moments are given by
\begin{align*}
m_k = \bE[X^k] = \sum_{l=0}^k \lambda^l \stirl{k}{l}
\end{align*}
where $ \stirl{k}{l}$ denotes the  Stirling numbers of the second kind, defined as
\begin{align*}
\stirl{k}{l} = \frac{1}{l!} \sum_{j=0}^l (-1)^{l-j} \binom{l}{j} j^k.
\end{align*}
In particular $\stirl{n}{1} = 1$ and $\stirl{n}{n} = 1$ for any $n \ge 1$. Moreover the sum over the first $k$ values of the Stirling numbers of the second kind gives  $\sum_{l=0}^k \stirl{k}{l}=B_k$, where $B_k$ is the $k$th Bell number. Hence, If $\lambda >1$, then $m_k \leq \lambda ^kB_k$, while if $\lambda <1$, then $m_k \leq  B_k$. This eventually yields
\begin{equation}\label{eq:st_Poisson}
m_k \leq ( \lambda ^k\vee 1) B_k\leq  ( \lambda ^k\vee 1)\left(\frac{0.792 k}{\log(k+1)}\right)^k
\end{equation}
\subsection{Power law distributions}\label{powerlaw}
Given  $\alpha \in (0,1)$, let us consider a sequence of random variables  $Y_{[m]}$ with density
\begin{align*}
f_m(y) = c_m \, y^{-\alpha-1} \, {\pmb 1}_{\left(\frac{1}{m},m^2\right)}(y) 
\end{align*}
where
\begin{align*}
c_m =  \frac{\alpha}{m^{\alpha}(1 - m^{-3\alpha})}.
\end{align*}
The moments of  $Y_{[m]}$ are given by
\begin{align*}
\bE[Y_{[m]}^k] =& c_m \, \int_{1/m}^{m^2} y^{k-\alpha-1} \, {\rm d}y = \frac{\alpha}{k-\alpha} \, \frac{1}{m^{\alpha}(1 - m^{-3\alpha})}
 \left( m^{2(k-\alpha)} - m^{-(k-\alpha)} \right)
 \\
 =& \frac{\alpha}{k-\alpha} m^{2k-3\alpha} \frac{1 -  m^{-3(k-\alpha)}}{1 - m^{-3\alpha}}
\end{align*}
and for every $m > 1$ and every $k \ge 1$ we obtain the  estimate
\begin{align}\label{eq:st_mom_Y_m-app}
\bE[Y_{[m]}^k]  \le c(\alpha) \, m^{2k-3\alpha}
\end{align}
where $c(\alpha) := 1 \vee \frac{\alpha}{1-\alpha}$.

%
%
%
%




\begin{thebibliography}{10}

\bibitem{AlMa}
S. Albeverio and S. Mazzucchi,
\newblock \emph{A unified approach to infinite-dimensional integration. }
\newblock Rev. Math. Phys. 28 (2016), no. 2, 1650005, 43 pp. 

\bibitem{allouba1} H. Allouba and W. Zheng. 
\emph{Brownian-time processes: The PDE connection and the halfderivative generator.}
Ann. Prob., 29, 1780--1795, 2001.

\bibitem{allouba2} H. Allouba. 
\emph{Brownian-time processes: The PDE connection II and the corresponding Feynman-Kac
formula.} 
Trans. Amer. Math. Soc., 354, 4627--4637, 2002.

\bibitem{BM01}
B.~Baeumer and M.~Meerschaert.
\newblock \emph{Stochastic solutions for fractional Cauchy problems}.
\newblock \emph{Fract. Calc. Appl. Anal.}, 4  (4):  481 -- 500,   2001.

\bibitem{BMN09}
B. Baeumer, M.M. Meerschaert and E. Nane. 
\emph{Brownian subordinators and fractional Cauchy problems.}
Trans. Amer. Math. Soc., 361, 3915--3930, 2009.

\bibitem{BaeMihKov2015}
B. Baeumer, M. Kov\'acs and H. Sankaranarayanan.
\emph{Higher order Gr\"{u}nwald approximations of fractional derivatives and fractional powers of operators.}
Trans. Amer. Math. Soc. 367, 813--834, 2015.

\bibitem{BARL}
M. T.~Barlow.
\newblock {Diffusions on fractals}.
\newblock {Lectures on Probability Theory and Statistics (Saint-Flour 	1995)},
\newblock {Volume 1690 of the series Lecture Notes in Mathematics 1060 pp. 1-121},
\newblock {Springer 1998}.


\bibitem{Balak60}
A.~Balakrishnan.
\newblock \emph{Fractional powers of closed operators and semigroups generated by
  them.}
\newblock \emph{Pacific J. Math.}, 10:\penalty0 419 -- 437, 1960.

\bibitem{Bazh00}
E.~G. Bazhlekova.
\newblock \emph{Subordination principle for fractional evolution equations.}
\newblock \emph{Frac. Calc. Appl. Anal.}, 3:\penalty0 213 -- 230, 2000.

\bibitem{BeHocOr}
S. Beghin, K. Hochberg, E. Orsingher.
\newblock \emph{Conditional maximal distributions of processes related to higher-order heat-type equations.}
\newblock  Stochastic Process. Appl.  85  (2000),  no. 2, 209--223.

\bibitem{OB03}
L.~Beghin and E.~Orsingher.
\newblock \emph{The telegraph process stopped at stable-distributed times and its
  connection with the fractional telegraph equation.}
\newblock \emph{Fract. Calc. Appl. Anal.}, 6:  187 -- 204, 2003

\bibitem{BergBoyDelau93}
C. Berg, Kh. Boyadzhiev, R. deLaubenfels. 
\emph{Generation of generators of holomorphic semigroups.}
{J. Austral. Math. Soc.} (Series A) 55, 246--269, 1993.

\bibitem{BeTa}
D.\ Berend and T.\ Tassa.  {\em Improved bounds on Bell numbers and on moments of sums of random variables}. Probability and Mathematical Statistics 30 (2):,185--205, 2010.


\bibitem{Ber97}
J. Bertoin. 
{Subordinators: examples and applications. }
In \emph{Lectures on probability theory and statistics (Saint-Flour, 1997)}, 1--91. Springer, Berlin, 1999.


\bibitem{BoMa15}
S. Bonaccorsi and S. Mazzucchi, 
\emph{ High order heat-type equations and random walks on the complex plane.}
 Stochastic Process. Appl. 125 (2), 797--818, 2015. 

\bibitem{BoCalMaz}
S. Bonaccorsi, C. Calcaterra and S. Mazzucchi, 
\emph{ An It\^o calculus for a class of limit processes arising from random walks on the complex plane.}
{ Stochastic Process. Appl.}, Available online 18 January 2017. doi = http://dx.doi.org/10.1016/j.spa.2016.12.009


\bibitem{Boch49}
S.~Bochner.
\newblock \emph{Diffusion equation and stochastic processes.}
\newblock {Proc. Nat. Acad. Sciences, U.S.A.}, 35:\penalty0 368 -- 370,
  1949.
  
\bibitem{Burdzy1995}
K.~Burdzy and A.~M{{\setbox0=\hbox{a}{\ooalign{\hidewidth
  \lower1.5ex\hbox{`}\hidewidth\crcr\unhbox0}}}}drecki.
\newblock \emph{An asymptotically {$4$}-stable process.}
\newblock In {\em Proceedings of the {C}onference in {H}onor of {J}ean-{P}ierre
  {K}ahane ({O}rsay, 1993)}, number Special Issue, pages 97--117, 1995.

\bibitem{DeBla}
R. D.~DeBlassie.
\newblock \emph{Iterated Brownian motion in an open set.}
\newblock {Ann. Appl. Probab.} Volume 14, Number 3 (2004), 1529--1558.

\bibitem{DalFo}
Yu. L. Daletsky and S. V. fomin.
\newblock  \emph{Generalized measures in function spaces.}
\newblock { Theory Prob. Appl.} 10 (2), 304-316, 1965.


\bibitem{Dyn}
E. B. Dynkin.
\newblock  Theory of Markov processes.
\newblock Dover Publications, Inc., Mineola, NY, 2006.

\bibitem{dovSPL}
M. D'Ovidio.
\emph{On the fractional counterpart of the higher-order equations.}
Statistics \& Probability Letters, 81 (12), 1929--1939, 2011.

\bibitem{DOVsl}
M.~D'Ovidio.
\newblock \emph{From Sturm-Liouville problems to fractional and anomalous diffusions. }
\newblock {Stochastic Processes and their Applications, 122, (2012), 3513--3544}.


\bibitem{Feller52}
W.~Feller.
\newblock \emph{On a generalization of Marcel Riesz' potentials and the semigroups
  generated by them}.
\newblock {Communications du seminaire mathematique de universite de Lund,
  tome supplimentaire}, 1952.

\bibitem{Fre}
M. Freidlin.
\newblock {Functional integration and partial differential equations.}
\newblock Princeton University Press, Princeton NJ (1985).


\bibitem{Funaki1979}
T.~Funaki.
\newblock \emph{Probabilistic construction of the solution of some higher order
  parabolic differential equation.}
\newblock {\em Proc. Japan Acad. Ser. A Math. Sci.}, 55(5):176--179, 1979.

\bibitem{GioRom92}
M.~Giona and H.~Roman.
\newblock \emph{Fractional diffusion equation on fractals: one-dimensional case and
  asymptotic behavior.}
\newblock \emph{J. Phys. A}, 25:  2093 -- 2105, 1992.


\bibitem{Gorenflo1997}
R.\ Gorenflo and F.\ Mainardi.
\emph{ Fractional calculus: integral and differential equations of fractional order.} 
In \emph{Fractals and fractional calculus in continuum mechanics (Udine, 1996)}, 223--276, CISM Courses and Lectures, 378, Springer, Vienna, 1997.

\bibitem{Hil00}
R.~Hilfer.
\newblock \emph{Fractional diffusion based on Riemann-Liouville fractional
  derivatives}.
\newblock {J. Phys. Chem. B}, 104:  3914 -- 3917, 2000.


\bibitem{HoWe72}
H.~H\"ovel and U.~Westphal.
\newblock \emph{Fractional powers of closed operators.}
\newblock {Studia Math.}, 42:\penalty0 177 -- 194, 1972.

\bibitem{Hoc}
K.~J. Hochberg.
\newblock \emph{A signed measure on path space related to {W}iener measure.}
\newblock { Ann. Probab.}, 6(3):433--458, 1978.

\bibitem{Hoc8}
K. J. Hochberg and E. Orsingher. {\em The arc-sine law and its analogs for processes governed
by signed and complex measures.} Stochastic Process. Appl., 52(2):273--292, 1994

\bibitem{Kac1} M. Kac.
\emph{On distributions of certain Wiener functionals.}
{Trans. Am. Math. Soc.} {\bf 65}, 1-13, 1949.


\bibitem{Kac2} M. Kac.
{Integration in function spaces and some of its applications.}
Lezioni Fermiane. [Fermi Lectures] Accademia Nazionale dei Lincei, Pisa, 1980.


\bibitem{KeyaLiza12} V. Keyantuo, C. Lizama.
\emph{On a connection between powers of operators and fractional Cauchy problems.}
{J. Evol. Equ.}, 12, 245--265, 2012.

\bibitem{KexJigJun2012}
L. Kexue and P. Jigen and J. Junxiong.
\emph{Cauchy problems for fractional differential equations with Riemann-Liouville fractional derivatives.}
{Journal of Functional Analysis}  263 (2), 476--510, 2012.


\bibitem{Koc89}
A.~N. Kochubei.
\newblock \emph{The Cauchy problem for evolution equations of fractional order}.
\newblock {Differential Equations}, 25:  967 -- 974, 1989.

\bibitem{Koc90}
A.~N. Kochubei.
\newblock \emph{Diffusion of fractional order.}
\newblock {Lecture Notes in Physics}, 26:  485 -- 492, 1990.


\bibitem{Kom66}
H.~Komatsu.
\newblock \emph{Fractional powers of operators.}
\newblock {Pacific J. Math.}, 19:\penalty0 285 -- 346, 1966.

\bibitem{KraSob59}
M.~A. Krasnosel'skii and P.~E. Sobolevskii.
\newblock \emph{Fractional powers of operators acting in Banach spaces}.
\newblock {Doklady Akad. Nauk SSSR}, 129:\penalty0 499 -- 502, 1959.


\bibitem{KraPar}
S. Krantz and H.  Parks.
{A Primer of Real Analytic Functions}.
 Birkh\"auser Verlag, Boston (2002).


\bibitem{Kry}
V.~J. Krylov.
\newblock \emph{Some properties of the distribution corresponding to the equation
  {$\partial u/\partial t=(-1)^{q+1} \partial ^{2q}u/\partial x^{2q}$}.}
\newblock {\em Soviet Math. Dokl.}, 1:760--763, 1960.

\bibitem{Lac11}
A. Lachal. {\em Distributions of sojourn time, maximum and minimum for pseudo-processes
governed by higher-order heat-type equations.} Electron. J. Probab., 8:no. 20, 53 pp. (electronic),
2003.

\bibitem{Lachal2013}
A. Lachal. {\em From Pseudorandom Walk to Pseudo-Brownian Motion: First Exit Time from a One-Sided or a Two-Sided Interval.}
International Journal of Stochastic Analysis, v.\ 2014, Article ID 520136, 49 pages, 2014.

\bibitem{LevLyo}
D. Levin, T. Lyons.
\newblock \emph{A signed measure on rough paths associated to a PDE of high order: results and conjectures.}
\newblock Rev. Mat. Iberoam. 25 (2009), no. 3, 971-994.


\bibitem{BMN09ann}
M.~Meerschaert, E.~Nane, and P.~Vellaisamy.
\newblock \emph{Fractional Cauchy problems on bounded domains}.
\newblock {Ann. Probab.}, 37  (3):  979 -- 1007, 2009.

\bibitem{meerschaert-nane-xiao}
 M.M.~Meerschaert, E. ~Nane, Y.~Xiao,  
 \newblock \emph{Fractal dimensions for continuous time random walk limits}.
 \newblock  {Statist. Probab. Lett.}, 83 (2013) 1083--1093.


\bibitem{MSheff04}
M.~Meerschaert and H.~P. Scheffler.
\newblock \emph{Limit theorems for continuous time random walks with infinite mean
  waiting times.}
\newblock {J. Appl. Probab.}, 41:  623 -- 638, 2004.

 \bibitem{Meerschaert2013}
 Meerschaert, M. M. and Straka, P.,
\emph{Inverse stable subordinators},
{Mathematical Modelling of Natural Phenomena} 8(2), 1--16, 2013.
  
\bibitem{Miller1993}
K.\ Miller and B.\ Ross. {An introduction to the fractional calculus and fractional differential equations.} John Wiley \& Sons, 1993.

\bibitem{Nane2008}
E. Nane. 
\emph{Higher order PDE's and iterated processes.}
{Trans. Amer. Math. Soc.} 360: 2681--2692, 2008. 

\bibitem{NANERW}
E.~Nane.
\newblock \emph{Fractional Cauchy problems on bounded domains: survey of recent
  results}.
\newblock  Fractional dynamics and control, 185–198, Springer, New York, 2012.

\bibitem{Nig86}
R.~Nigmatullin.
\newblock \emph{The realization of the generalized transfer in a medium with fractal
  geometry.}
\newblock {Phys. Status Solidi B}, 133:  425 -- 430, 1986.

\bibitem{Nik15}
Ya. Yu. Nikitin and E. Orsingher. {\em On sojourn distributions of processes related to some
higher-order heat-type equations.} J. Theoret. Probab., 13(4):997--1012, 2000.

\bibitem{Ors19}
E. Orsingher. {\em Processes governed by signed measures connected with third-order ?heat-type?
equations.} Lithuanian Math. J., 31(2):220--231, 1991.


\bibitem{Nishioka1996}
K. Nishioka. {\em Monopoles and dipoles in biharmonic pseudo-process.} Proc. Japan Acad. Ser. A Math. Sci., 72(3):47--50, 1996.

\bibitem{Nishioka2001}
K. Nishioka. {\em Boundary conditions for one-dimensional biharmonic pseudo process.} Electron. J. Probab., 6:no. 13, 27 pp. (electronic), 2001.

\bibitem{Orsingher2009}
E.\ Orsingher and L.\ Beghin.
\emph{ Fractional diffusion equations and processes with randomly varying time.} Ann. Probab. 37 (1), 206--249, 2009.

\bibitem{OrsDov2015}
E. Orsingher and M. D'Ovidio.
\emph{Higher-Order Laplace Equations and Hyper-Cauchy Distributions.}
{J. Theor. Probab.}, 28, 2015.

\bibitem{OrTo} E. Orsigher and B. Toaldo,
\emph{Pseudoprocesses related to space-fractional higher order heat-type equations.}
{Stochastic Analysis and Applications}, 32, 619--641, 2014.


\bibitem{Phill52}
R. S. Phillips.
\emph{On the generation of semigroups of linear operators.}
J. Math., 2 (3), 343--369, 1952.


\bibitem{SKM}
S. G. Samko, A. A. Kilbas and O. I. Marichev.
\emph{Fractional Integrals and Derivatives: Theory and Applications. }
Gordon and Breach, Yverdon, 1993.



\bibitem{SWyss89}
W.~Schneider and W.~Wyss.
\newblock \emph{Fractional diffusion and wave equations.}
\newblock {J. Math. Phys.}, 30:  134 -- 144, 1989.



\bibitem{Shill96}
R.L. Schilling.
\emph{On the domain of the generator of a subordinate semigroup}, in: J. Kr\'{a}l, et al. (eds.),
\emph{Potential Theory-ICPT 94.} Proceedings Internat. Conf. Potential Theory, Kouty (CR), 1994 (de
Gruyter, Berlin, 1996), pp. 449-462.



\bibitem{Bern}
R. L. Schilling, R. Song, Z. Vondracek.	
{Bernstein Functions: Theory and Applications.}
{Walter de Gruyter}, 2010



\bibitem{Wata61}
J.~Watanabe.
\newblock \emph{On some properties of fractional powers of linear operators.}
\newblock {Proc. Japan Acad. Ser. A Math. Sci.}, 37:\penalty0 273 -- 275,
  1961.


\bibitem{Wyss86}
W.~Wyss.
\newblock \emph{The fractional diffusion equations.}
\newblock {J. Math. Phys.}, 27:  2782 -- 2785, 1986.


\bibitem{Tho}
E. Thomas, 
\newblock \emph{Projective limits of complex measures and martingale convergence.}
\newblock { Probab. Theory Related Fields } {\bf 119} (2001), no.4, 579-588

\bibitem{Zasl94}
G.~Zaslavsky.
\newblock \emph{Fractional kinetic equation for Hamiltonian chaos}.
\newblock {Phys. D}, 76:  110 -- 122, 1994.

\end{thebibliography}
\end{document}